\documentclass[11pt,a4paper]{article}

\usepackage[pdftitle={Paper}, pdfauthor={}]{hyperref}

\usepackage{amsxtra,amssymb,amsmath,amsthm,amsbsy,enumerate}
\usepackage[all,2cell,v2]{xy}\UseTwocells\UseHalfTwocells
\usepackage{derivative}
\usepackage{comment}
\usepackage{xcolor}
\usepackage{graphicx}

 \theoremstyle{definition}
\newtheorem{theorem}	{Theorem}[subsection]
\newtheorem{lemma}	[theorem]{Lemma}
\newtheorem{proposition}[theorem]{Proposition}
\newtheorem{remark}	[theorem]{Remark}
\newtheorem{example}	[theorem]{Example}
\newtheorem{corollary}	[theorem]{Corollary}
\newtheorem{definition}	[theorem]{Definition}

\newtheorem*{theorem-1}	{Theorem \ref{thm:1} (Frames)}
\newtheorem*{theorem-2}	{Theorem \ref{thm:2} (Differentiation)}
\newtheorem*{theorem-3}	{Theorem \ref{thm:3} (van Est)}
\newtheorem*{theorem-4}	{Theorem \ref{thm:4} (Abstract differentiation)}

\def\id{{\rm id}}

\def\C{{\mathcal C}}

\def\M{{\mathcal M}}

\def\N{{\mathbb N}}
\def\R{{\mathbb R}}
\def\Z{{\mathbb Z}}

\def\Csev{{\mathcal C}^{Sev}}

\def\h{\mathfrak{h}}

\def\xto{\xrightarrow}
\def\to{\rightarrow}

\def\toto{\rightrightarrows}
\def\xfrom{\xleftarrow}

\def\then{\Rightarrow}

\def\mono{\rightarrowtail}

\def\J{\hat J}
\def\JJ{{\mathcal J}}

\def\Jgen{J}

\def\d{\mathbf{d}}

\def\T{{T^{[1]}}}

\def\nameD{interpolating double complex}

\def\colim{{\rm colim}\,}

\newcommand{\Ale}[1]{\texttt\tiny \marginpar{\textcolor{blue}{X A}} \textcolor{blue}{\texttt{\small #1}}}

\def\ve{\mathfrak{ve}}

\def\cosk{{\rm cosk}}
\def\sk{{\rm sk}}

\def\cat#1{{\rm #1}}

\def\supp{{\rm supp\,}}
\def\ord{{\rm ord\,}}

\def\O{{\mathcal O}}

\def\dtot{d^{tot}}

\newcommand{\un}[1]{\underline{#1}}

\newcommand{\comentar}[1]{\textcolor{red}{-- Here there is hidden text (modify command "comentar" tex-file) -- }}

\begin{document}

\title{\bf Geometric differentiation of simplicial manifolds}

\author{Alejandro Cabrera \and Matias del Hoyo}

\date{\today}

\maketitle

\abstract{
%
We provide a complete geometric solution to the problem of differentiating simplicial manifolds,
extending classical Lie theory and complementing existing homotopical and formal approaches within
a unifying framework.
First, we establish a normal form theorem setting a system of compatible tubular neighborhoods.
Building on this description, we
identify a differentiating ideal in the algebra of cochains, prove that the quotient is semi-free, and
interpret it as the Chevalley–Eilenberg algebra of the thus defined  higher Lie algebroid.
%
%
As an application, we introduce a higher version of the van Est map and prove a van Est isomorphism theorem in cohomology, under natural connectivity assumptions. 
Finally, we identify the algebraic mechanism underlying geometric
differentiation as a monoidal refinement of the dual Dold–Kan correspondence, providing a conceptual explanation of the construction and relating it to earlier homotopical and functor-of-points
approaches.
}

\tableofcontents


\section{Introduction}


Higher differential geometry is a rapidly developing area, paralleling the work of Lurie, Toen, and others in algebraic geometry. It uses homotopical and homological methods to study the geometry of manifolds and related spaces, providing a natural framework for describing symmetries defined up to isomorphisms or equivalences.  In recent years, the area has been strongly influenced by interactions with symplectic geometry, algebraic topology, and mathematical physics.
%
%
%
Two fundamental objects play a central role in this setting. On the global side, one finds simplicial manifolds and, in particular, higher Lie groupoids. They generalize Lie groups and groupoids by encoding higher compositions through horn-filling conditions rather than strict multiplication laws. 
On the infinitesimal side, one finds differential graded manifolds, also known as higher Lie algebroids or $NQ$-manifolds.
They allow coordinates of positive degree, their functions form differential graded algebras, encoding higher brackets generalizing
Lie algebras and Lie algebroids.

A central problem of the area is to elucidate the interplay between these two frameworks via
differentiation and integration, generalizing classical Lie theory.
The case over a point is thoroughly studied in  \cite{getzler09,henriques,R-NV20} and the more recent preprints  \cite{rw,Rog25}. 
The present paper deals with the problem of differentiating higher Lie groupoids and general simplicial manifolds. 
Important contributions in this direction follow a formal supergeometric functor-of-points approach by Severa \cite{Severa}, further explored in \cite{Li, LRWZ, dorsch}, or an abstract derived geometry approach by Pridham and others \cite{Nui-thesis, Pri10, Pri20}, or focus on the particular order $2$ cases \cite{angulo-cueca, L-BS}. 
In this work we develop a general, simple, geometric differentiation in higher Lie theory, complementing and unifying the above approaches, and providing a complete and independent answer to the problem.



Our differentiation functor, from simplicial manifolds to higher Lie algebroids, is developed in
a way that closely parallels classical Lie theory, keeping the first nontrivial terms in the Taylor
expansions of the zeroth face map, which embodies the higher multiplications. Within our theory,
graded commutativity, freeness, and the overall higher algebraic structures emerge naturally at
the infinitesimal level, from the interaction between truncation, normalization, and the simplicial
structure. Moreover, our approach also leads to a higher van Est theory allowing for differentiation and integration of cochains and forms.

\medskip

{\bf Main contributions.} 
Starting from a simplicial manifold $(G_n,d_i,s_j)$, the first step consists in gaining local control of the simplicial structure near the space of objects $G_0$ by working with a suitable system of tubular neighborhoods which we call a \emph{frame} (Definition \ref{def:frame}). We regard each pair $(G_n,G_0)$ as a microbundle in the sense of Milnor via the projection onto the first vertex $e:G_n\to G_0$.
The positive faces and the degeneracies preserve this projection, yielding linear maps between the microbundle tubular nieghborhoods $(\nu(G_n,G_0),G_0,\pi)$. We inductively establish local coordinates around $G_0$ across all simplicial degrees.

\begin{theorem-1}
Every simplicial manifold $(G_n,d_i,s_j)$ admits a frame $\phi=(\phi_n)_{n\geq1}$, namely a family of microbundle tubular neighborhoods  $\phi_n:(\nu(G_n,G_0),G_0,\pi)\to(G_n,G_0,e)$
compatible with the positive faces and the degeneracies,
$d_i\phi_n=\phi_{n-1}d_i$ for $i>0$, and $s_j\phi_n=\phi_{n+1}s_j$.
\end{theorem-1}

\smallskip

In the second step, we use a frame to compare smooth functions on $G_n$ with polynomial functions on the fibers of the normal bundle, work with simplicial cochains $C(G)$, and naturally arrive at our definition of differentiation. Inside the normalized cochain algebra $C_N(G)$, we identify an ideal $J$ encoding higher-order vanishing along the degeneracy locus, 
the information that disappears under differentiation. We call its closure $\hat J=J+\delta(J)$ the \emph{differentiating ideal} (Definition \ref{def:diffideal}). The quotient is shown to be semi-free and therefore to define a higher Lie algebroid structure. 

\begin{theorem-2}
Let $G$ be a simplicial manifold, $A_G$ its tangent complex, and $\J\subset C_N(G)$ its differentiating ideal.
 Then $C_N(G)/\J$ is a semi-free commutative differential graded algebra with indecomposables $\Gamma(A^*_G)$, defining a higher Lie algebroid structure on $A_G$. A frame $\phi$ for $G$ yields a splitting $\phi_\star:S(\Gamma A^*)\overset{\sim}{\to} 
C_N(G)/\J$.
\end{theorem-2}

The quotient dga $C_N(G)/\J$ is called the Chevalley-Eilenberg algebra of the resulting higher algebroid, and denoted $CE(A_G)$.
A straightforward extension of the previous result allows us to differentiate forms instead of functions, by working with the  Bott--Shulman complex of a simplicial manifold, and obtaining the Weil algebra of the associated higher Lie algebroid at the infinitesimal level, see Corollary \ref{cor:thm2forms}. 

Third, we establish a van Est isomorphism theorem for higher Lie groupoids. Our approach comes with a natural van Est map relating global and infinitesimal cochains, which is simply given by the quotient map
$$
\ve:C_N(G)\to CE(A_G)=C_N(G)/\hat J,
$$
and which restricts to the classical one for Lie groups and groupoids (see Proposition \ref{prop:gdcomparison}). In those cases, the classical van Est theorem (\cite{vanEst, crainic}), asserts that $\ve$ induces an isomorphism in cohomology in a certain range and under suitable topological assumptions. Within our framework, we introduce and discuss a notion of $n$-connectivity for simplicial manifolds (Definition \ref{def:connected}), expressed in terms of the fibers of the face maps, and we achieve a geometric proof for a higher van Est theorem, built using Illusie's Decalage (\cite{illusie}). 

\begin{theorem-3}
Let $G$ be a higher Lie groupoid which is $n$-connected. Then, the van Est map induces isomorphism in cohomology $H^k(G)\to H^k(A_G)$ for every $k \leq n$.
\end{theorem-3}

\noindent A version of this theorem restricted to certain minimal augmented dga's appeared in \cite[Thm. 8.1'(v)]{Sullivan} and, in the context of derived stacks, in \cite{Nui-thesis}. The above yields a complete,  concrete and strict generalization of van Est's isomorphism to the higher context, allowing to integrate infinitesimal cocycles under the suitable connectivity assumptions. As an application, our concrete treatment allows us to provide a full Lie theoretic characterization of shifted symplectic structures on higher groupoids (Corollary \ref{cor:shiftedsympl}).

The last section of the paper identifies the algebraic mechanism underlying the differentiation of simplicial manifolds and higher Lie groupoids. We show that differential graded algebras arise as the infinitesimal objects inside the category of cosimplicial algebras, and that differentiation can be characterized as the left adjoint to monoidal denormalization. 
We say that a cosimplicial algebra $X$ is \emph{infinitesimal} if the simplicial isomorphism $X\cong (X_N)_K$ preserves the algebra structure. We find a canonical obstruction abstracting higher vanishing, namely the products $\mu(\mathcal J_X)$ of overlapping tensors $\mathcal J_X\subset (X\otimes X)_N$, which span the differentiating ideal $\hat J_X$ studied before. By quotienting by this ideal, one obtains a functorial and computable differentiation procedure.

\begin{theorem-4}
The abstract differentiation $N'(X)=X_N/\J_X$
is left adjoint to the denormalization functor, making differential graded algebras a reflective subcategory of cosimplicial algebras.
$$N':c\cat{Alg}\leftrightarrows dg\cat{Alg}:K, \qquad N'\dashv K.$$
Moreover, $N'$ preserves commutativity, and for a simplicial manifold $G$ it recovers the geometric differentiation as $N'(C(G))\cong CE(A_G)$.
\end{theorem-4}

This generalizes naturally to super algebras, and both explains and unifies earlier constructions, recovering the Chevalley-Eilenberg and Weil algebras uniformly. It can also be seen as a generalization of the main result on Beilinson's small algebras, see \cite[\S 7]{burgos}.  
Building on our Theorem, we show in Section \ref{subsec:severa} that our concrete construction defines a differential graded manifold solving Severa's representability problem \cite{Severa}. The paper \cite{LRWZ} previously obtained the underlying graded manifold by different methods.
Additionally, this theorem shows that the idea of modding out by higher vanishing recovers the construction already proposed by Pridham in derived geometry \cite{Pri10,Pri20} and further developed by Rogers in the non-derived context. 

\medskip


{\bf Outlook and future directions.}
Our work has the potential to unlock several future investigations. We now list some of the future goals that we plan to explore elsewhere, as well as general open questions related to the project.

\noindent\emph{- Double structures}: The differentiation theory developed in this paper admits natural extensions beyond the simplicial setting, in particular, to bisimplicial manifolds $G_{\bullet,\bullet}$. In such a case, our differentiation can be applied first vertically and then horizontally $A^hA^v(G)$, viceversa $A^vA^h(G)$ or all at once $A(Tot(G))$ after taking the total simplicial structure given by the Artin-Mazur codiagonal $Tot(G)$ \cite{artin-mazur,cegarra} (see also \cite{mt11}). Our concrete approach to differentiation allows for a direct comparison and the expected result is that the three induce the same infinitesimal object, in the sense that the total (or \emph{Weil} \cite{mepi21,vor02}) algebras $Tot(A^hA^v(G))\simeq Tot(A^vA^h(G))$ associated with the doubles are naturally isomorphic to $A(Tot(G))$.
This perspective offers a conceptual explanation for Weil-type constructions and for their functoriality \cite{angulo-cueca,bcdh,mepi21}, and suggests a general framework for multi-directional differentiation. For example, the nerve of the pair groupoid of a simplicial manifold $G$ should have total second order differentiation being the Weil algebra $W(A_G)$. A detailed account of iterated differentiation and its geometric consequences will be developed elsewhere.

\noindent\emph{- Lie theory for shifted geometric structures}: Extending Corollary \ref{cor:shiftedsympl} below for shifted symplectic structures, we plan to elucidate the differentiation and integration of shifted Poisson structures and other geometries (see a general review in \cite{cumava}). We think that our approach to differentiation can provide explicit and computable descriptions of structures that are often presented in a highly abstract homotopical language. Other applications of the van Est isomorphism theorem will also be explored, including higher rigidity problems.

\noindent\emph{- Integration problems}: 
Finite-dimensional integration theory is developed in some special cases, such as for Lie 1-algebroids \cite{cf}, the cases over a point \cite{getzler09,henriques,rw,Rog25}, and for local integrations \cite{sesi}. Our differentiation theory allows for a systematic exploration of the problem of finite-dimensional integrations of general higher Lie algebroids. Of particular interest are integrations of Courant algebroids (\cite{L-BS,mt11,SZ17}), singular foliations, representations up to homotopy (\cite{ac,gsm,dht}) and other examples arising in mathematical-physics.

Finally, another line in which we plan to apply our framework is that of higher stacks, where it can allow for fine, strict results towards the understanding of higher differentiable stacks, especially for moduli spaces solving problems of classifications up to equivalences.

\smallskip

{\bf Recent related work.}
During the final stages of the elaboration of this paper, the preprint \cite{Rog25} appeared, developing an alternative yet related approach to differentiation, specialized to the particular case of $\infty$-groups. Among the relations, the study of frames is related to certain PBW bases in the context of formal neighborhoods \cite[Thm. 2.49]{Rog25}, and an alternative description of the left adjoint to denormalization in the commutative setting appears in \cite[Thm. 3.9]{Rog25}. 

\smallskip


{\bf Organization.}
The paper is organized as follows. Section 2 revolves around frames, which are systems of microbundle tubular neighborhoods for simplicial manifolds, and proves the frame theorem \ref{thm:1}. Section 3 is the main one, it develops the local study around the units and concludes with the main differentiation theorem \ref{thm:2}. Section 4 establishes the van Est isomorphism theorem \ref{thm:3}. The final section 5 studies the abstract differentiation theorem \ref{thm:4} and explains the relation with other approaches in the literature.


\bigskip

{\bf Acknowledgments.} The authors would like to thank H. Bursztyn, M. Cueca, R.L. Fernandes, E. Meinrenken, C. Rogers, P. Severa and C. Zhu for useful discussions and exchanges. They are particularly grateful to C. Rogers and C. Zhu for discussions of related literature and their feedback on the first version.
A.C.’s research was partially supported by the grants CNPq
PQ 309847/2021-4, CNPq Universal 402320/2023-9 and
FAPERJ CNE E-26/204.097/2024.
M.dH. was partially supported by the grants CNPq
PQ 310289/2020-3 and CNPq Universal 402320/2023-9.


\section{Frames on simplicial manifolds}

In this section, using the language of Milnor microbundles, we develop a local normal form for simplicial manifolds, showing in Theorem \ref{thm:1} that we can build a system of tubular neighborhoods on which we simultaneously linearize the degeneracy and face maps except for $d_0$.



\subsection{From Lie groupoids to simplicial manifolds}


This subsection sets our notations and conventions, presents simplicial manifolds as our main object of study, and reviews  higher Lie groupoids via the smooth horn-filling condition.


\begin{remark}
As usual, $\Delta$ is the category of finite ordinals $[n]=\{0,1,\dots,n\}$ and non-decreasing order maps $\theta:[m]\to[n]$. We write $\iota=\iota_n:[n]\to[n]$ for the identity, $\delta_i:[n-1]\to[n]$ for the injection missing $i$ in the target and $\sigma_j:[n+1]\to[n]$ for the surjection repeating $j, \ j+1$ from the domain. 
A simplicial object in a category $\C$ is a contravariant functor $X:\Delta^\circ\to\C$, or equivalently, a system $(X_n,d_i,s_j)$ where $X_n$, $n\geq0$, are objects in $\C$, equipped with faces $d_i=\delta_i^*:X_n\to X_{n-1}$ and degeneracies $s_j=\sigma_j^*:X_n\to X_{n+1}$, $0\leq i,j\leq n$, satisfying the simplicial identities.
\[ d_jd_i=d_id_{j+1}, s_is_j=s_{j+1}s_i, \ i\leq j; d_is_j=s_{j-1}d_i, i<j; d_js_j=d_{j+1}s_j=1; d_is_j=s_jd_{i-1}\text{ otherwise.}  \] 
\end{remark}


\begin{definition}
A {\bf simplicial manifold} is a contravariant functor $G:\Delta^\circ\to \text{Man}$, or equivalently, a system $(G_n,d_i,s_j)$ where $G_n$ are manifolds, $d_i=G_n\to G_{n-1}$ are face maps and $s_j=G_n\to G_{n+1}$ degeneracy maps satisfying the simplicial identities. 
The {\bf (augmented) tangent complex} $(A_G,\partial)$ is the chain complex corresponding to $TG|_{G_0}$ under Dold-Kan:
$$A_n=
\bigcap_{0<i\leq n}\ker\big( (d_i)_*:TG_n\to TG_{n-1}\big)|_{G_0},\ n\geq 1 \qquad A_0=TG_0 \qquad  
\partial=(d_0)_*:A_n\to A_{n-1}$$ 
\end{definition}


Lie groupoids and Lie algebroids are fundamental objects in modern differential geometry. Lie groupoids are internal groupoids in the category of manifolds, they provide useful tools to describe local and global symmetries, and they are the global counterpart of Lie algebroids, which are central objects in the theory of foliations, Poisson geometry, and mathematical physics. Lie groupoids are our main motivating example of simplicial manifolds via the nerve functor, see e.g. \cite{bdh,survey}.

\begin{example}
If $H\toto M$ is a Lie groupoid, its {\bf nerve} $G=Ner(H)$ is the simplicial manifold with $G_0=M$, $G_1=H$, and $G_n=H\times_M\dots\times_MH$ are chains of $n$ composable arrows in $H$. If 
$$g=(x_n\xfrom{g_n} x_{n-1}\xfrom{g_{n-1}}\cdots \xfrom{g_1}x_0)$$
is an $n$-simplex in $G$, the face $d_i$ deletes $x_i$ and either composes the adjacent arrows or disregards the extreme one, while the degeneracy $s_j$ repeats $x_j$ by inserting an identity.
\end{example}


If $X$ is a simplicial set, $x\in X_n$ is degenerate if it is in the image of some $s_j$. Any simplex can be uniquely written as $x=\theta^*(y)$ with $\theta:[n]\to[k]$ surjective and $y\in X_k$ non-degenerate. The {\bf totally degenerated} $n$-simplices are $\pi^*(X_0)\subset X_n$, where $\pi:[n]\to[0]$, and give a canonical way to see $X_0$ as a subset of $X_n$.
We write $S_{n,j}=s_{j-1}(X_{n-1})\subset X_n$ and refer to $S_n=\bigcup_j S_{n,j}\subset X_n$ as the {\bf degeneracy locus}. If $G$ is a simplicial manifold, then each $S_{n,j}\subset G_n$ is an embedded submanifold, since $s_{j-1}:G_{n-1}\to G_n$ has $d_{j}$ as a left-inverse. Our first main theorem will show that, at least near $G_0$, the submanifolds $S_{n,j}$ intersect cleanly.


\begin{remark}
The $n$-simplex $\Delta^n$ is the simplicial set represented by $[n]$, so $\Delta^n_k=\{[k]\to[n]\}$. If $X$ is a simplicial set, then Yoneda sets natural identifications $Hom(\Delta^n,X)\cong X_n$, $f\mapsto f(\iota)$. The non-degenerate simplices of $\Delta^n$ are the injections $[k]\to[n]$. The boundary $\partial\Delta^n\subset \Delta^n$ is spanned by all the faces $\delta_i=d_i(\iota)\in \Delta^n_{n-1}$, and the $(n,k)$-horn $\Lambda^n_k\subset \partial\Delta^n$ is spanned by the faces $\delta_i$ except $\delta_k$.
\end{remark}

The modern approach to higher category theory replaces the existence of higher compositions by a horn-filling condition: a collection of simplices is composable if they are all but one face of a higher simplex, and the remaining face is a composition of them. Within the smooth setting, this gives the following working definition.

\begin{definition}
A {\bf higher Lie groupoid} is a simplicial manifold $(G,d_i,s_j)$ satisfying the following two conditions:
\begin{enumerate}[HG1)]
    \item The horn-space $G_{n,k}=Hom(\Lambda^n_k,G)=\{(g_i,i):d_i(g_j)=d_{j-1}(g_i) \ \forall i<j\}\subset\prod_{i\neq k} G_{n-1}$ is an embedded submanifold, and 
    \item the horn map $d_{n,k}=\prod_{i\neq k} d_i:G_n\to G_{n,k}$ is a surjective submersion for every $n,i$.
\end{enumerate}
A higher Lie groupoid is a {\bf Lie $m$-groupoid} if $d_{n,k}$ is 1-1 for every $n>m$.
\end{definition}

\begin{example}
Constant simplicial manifolds are the same as Lie 0-groupoids, while Lie 1-groupoids are exactly the same as nerves of Lie groupoids, according to the classical definition, see e.g. \cite{mm-book}. Lie 2-groupoids, as defined above, are a weaker version of the strict Lie 2-groupoids studied for instance, in \cite{dhs}. Other relevant examples of higher Lie groupoids arise in the theory of gerbes, of representations up to homotopy, and of shifted symplectic and Poisson structures (see \cite{cumava} for an overview).
\end{example}



\begin{remark}
Lie groupoids up to Morita equivalence serve as models for differential stacks, solutions to moduli problems of classifying objects up to isomorphisms. Generalizing this, higher Lie groupoids up to hypercovers model higher differentiable stacks (see \cite{getzler}),  which play a key role in the emerging derived differential geometry (see \cite{Nui-thesis,PTVV13,Pri20} and references therein), and appear naturally as moduli spaces for classifications up to equivalence. 
\end{remark}




\subsection{Microbundles, tubular neighborhoods and frames}


We introduce here the notion of a frame of a simplicial manifold as a system of compatible tubular neighborhoods, within the language of Milnor's microbundles.


\begin{remark}[Dold-Kan I]\label{rmk:dk1}
Fix $M$ a manifold. A simplicial vector bundle $(V\to M,d_i,s_j)$ yields a {\bf normalized complex} $(V_N,\partial)$, given by $(V_N)_n=\ker(d_{n,0}:V_n\to V_{n,0})=\bigcap_{i>0} \ker (d_i:V_n\to V_{n-1})$ and $\partial =d_0\vert_{(V_N)_n}$.
By the Dold-Kan Theorem, $N$ establishes an equivalence of categories between simplicial vector bundles and connective (i.e. non-negatively graded) chain complexes. The smooth vector bundle version is discussed, for instance, in \cite{dht}. 
$$N:sVB(M)\to Ch_{\geq0}(VB(M))$$
A key remark is that every simplicial vector bundle is a higher groupoid, namely $V_{n,k}\subset \prod_{i\neq k}V_{n-1}$ is a well-defined sub-bundle and $V_n\to V_{n,k}$ is an epimorphism.
\end{remark}

\begin{example}\label{ex:normal}
Starting with $G$ a simplicial manifold,  there is a short exact sequence of simplicial vector bundles over $G_0$, where $TG_0$ is regarded as a constant simplicial object:
$$0\to TG_0 \to TG\vert_{G_0}\to \nu(G,G_0)\to 0$$
The normalization $N(\nu(G,G_0))$ is the tangent complex of $G$, while $N(TG\vert_{G_0})$ is the augmented tangent complex. 
If $G$ is the nerve of a Lie groupoid, then $A_1=A(G)$ is the vector bundle of the underlying Lie algebroid and $\partial:A_1=A(G)\to A_0=TG_0$ is the anchor map.
\end{example}


Let $G=(G_n,d_i,s_j)$ be a simplicial manifold, embed $G_0\subset G_n$ via the totally degenerated simplices, and write $G_{n,k}\subset \prod_{i\neq k} G_{n-1}$ for the horn spaces, which are just subspaces in general, and $d_{n,k}:G_n\to G_{n,k}$ for the horn maps. 
We are interested in the geometry of each $G_n$ around $G_0$, so we shall work with germs $(M,G_0)$, where $M$ is a manifold and $G_0\subset M$ is embedded, and germ maps extending the identity $f:(M,G_0)\to(N,G_0)$. 
These germs are used in \cite{L-BS} and references therein with the name of \emph{microfolds}.
We say that such a map is an {\bf epimorphism} if $Df\vert_{G_0}$ is surjective, or equivalently, it can be represented by a surjective submersion, and that it is a {\bf monomorphism} if $Df\vert_{G_0}$ is injective, or equivalently, it can be represented by a closed embedding.
As proven in \cite[Thm 3.8]{dorsch}, the horn-filling conditions defining higher Lie groupoids are always satisfied locally around $G_0$. 
Next, we give a simplified proof of this fact, as a local version of \cite[Prop 6.5]{dhos}.

\begin{proposition}
Given $G$ a simplicial manifold, and given $n,k$, the horn space $(G_{n,k},G_0)$ is a well-defined subgerm of the product $(\prod_{i\neq k}G_{n-1},G_0)$, and the restriction $d_{n,k}:(G_n,G_0)\to(G_{n,k},G_0)$ is a germ epimorphism.
\end{proposition}

\begin{proof}
We can argue inductively in $n$. The statement is clear for $n=1$, since $G_{1,k}\cong G_0$ and $d_k=d_{1,k}:G_1\to G_0$ has maximal rank around $G_0$, for they are retractions for the inclusion.

Assume now that the claim is true for every $m,k$, $n>m$. Consider a decreasing filtration
$$\Delta^n=F_a\supset \Lambda^n_k=F_{a-1}\supset\dots\supset F_r\supset\dots\supset F_0=\ast$$ 
on which each step is an elementary collapse $F_r=F_{r-1}\cup_{\partial\Lambda^{n_r}_{k_r}}\Delta^{n_r}$, $n_r<n$.
This leads to a tower of germ restriction maps
$$(G_n,G_0) \to (G_{n,k},G_0)\to \dots\to 
(G^{F_r},G_0)\to \dots\to (G_0,G_0)$$
where the $r$-th step $(G^{F_r},G_0)\to(G^{F_{r-1}},G_0)$, $r<a$, is the base-change of the surjective submersion $(G_{n_r},G_0)\to (G^{\Lambda^{n_r}_{k_r}},G_0)$:
$$\xymatrix{ 
(G^{F_r},G_0) \ar[r] \ar[d] & 
(G^{F_{r-1}},G_0) \ar[d] \\
(G_{n_r},G_0) \ar[r] & (G^{\Lambda^{n_r}_{k_r}},G_0)}$$
Every surjective submersive germ map can be represented by a surjective submersion. Then, by the standard transversality criterion, we have that $(G^{F_r},G_0)$ is a smooth well-defined germ and $(G^{F_r},G_0)\to(G^{F_{r-1}},G_0)$ is surjective submersive, in particular $(G^{F_{a-1}},G_0)=(G_{n,k},G_0)$. It only remains to show that $(G_n,G_0)\to (G_{n,k},G_0)$ is a germ epimorphism. For this we observe that, at the linear level of normal bundles, this property holds since $\nu(G,G_0)$ is a simplicial vector bundle and hence satisfies the horn filling conditions (see Remark \ref{rmk:dk1}). Together with the fact that $\nu(G_{n,k},G_0)\simeq \nu(G_n,G_0)_{n,k}$, the result easily follows.
\end{proof}



Given $(M,S)$ a germ, by a
{\bf germ tubular neighborhood} we mean a 
germ isomorphism 
$\phi:(\nu(M,S),S)\to(M,S)$
such that $\nu(\phi)=\id_{\nu(M,S)}$, where we are identifying 
$E\cong \nu(E,S)$, $v\mapsto \odv{}{t}_{t=0} t.v$ for a vector bundle $E\to S$. 
Starting with a simplicial manifold $G$, a {\bf simplicial tubular neighborhood} 
may be defined as a system of germ tubular neighborhoods which are compatible with all the faces and the degeneracies:
$$(\phi_n:(\nu(G_n,G_0),G_0)\to (G_n,G_0))_n \qquad \phi d_i=d_i\phi \quad \phi s_j=s_j\phi,$$
where the structure maps on $\nu(G_n,G_0)$ are given by the normal derivatives of those of $G$ (see Ex. \ref{ex:normal}).
Such a family rarely exists since the simplicial normal bundle $\nu(G,G_0)$ is completely encoded in the tangent complex $(A_G,\partial)$, while the germ $(G,G_0)$ may include non-linear information. In the Lie groupoid case, a simplicial tubular neighborhood would force the orbits to be $0$-dimensional and the Lie algebroid bracket to be trivial. 
Next, we establish weaker versions of simplicial tubular neighborhoods, within the language of microbundles, see e.g. \cite{milnor}.

\begin{definition}
A {\bf microbundle} $(M,S,r)$ is a germ $(M,S)$ together with a germ projection $r:(M,S)\to (S,S)$ extending $\id_S$. 
A microbundle map $f:(M,S,r)\to (M',S',r')$ is a germ map compatible with the projections, namely $r'f=f r$. 
\end{definition}



\begin{example}
Every germ $(M,S)$ admits a germ tubular neighborhood, and can therefore be upgraded to a microbundle, though in a non-canonical way.  Moreover, given $(M,S,r)$ a microbundle, by standard arguments, we can always build a {\bf microbundle tubular neighborhood}, namely one compatible with the projections 
$\phi:(\nu(M,S),S,\pi)\to (M,S,r)$.
\end{example}

Given $G$ a simplicial manifold, the $0$-vertex map $e:G_n \to G_0$ defines a natural microbundle $(G_n,G_0,e)$
which, via a tubular neighborhood, is isomorphic to the normal bundle
$$ (\nu(G_n,G_0),G_0,\pi) \simeq(G_n,G_0,e) .$$
Every \emph{positive face} and every degeneracy can be regarded as a microbundle map. This is not the case with the 0-face, for this does not preserve the zero vertex in general.  

\begin{definition}\label{def:frame}
Given $G$ a simplicial manifold, we define a {\bf frame} $\phi=(\phi_n)_{n\geq1}$ as a family of microbundle tubular neighborhoods 
$\phi_n:(\nu(G_n,G_0),G_0,\pi)\to(G_n,G_0,e)$
compatible with the positive faces and the degeneracies, namely
$d_i\phi_n=\phi_{n-1}d_i$ for $i>0$, and $s_j\phi_n=\phi_{n+1}s_j$.    
\end{definition}

A frame allows us to take all the structure maps $d_i$ and $s_j$ to a normal form, except for $d_0$. We will show that every simplicial manifold $G$ admits a frame, and therefore, to understand $G$ near the units, besides the tangent complex $A_G$, the only extra data we need to take care of is the germ of the 0th face map $d_0$. 

\begin{example}[Frames on a Lie groupoid]
A frame on a Lie groupoid $G\toto M$, in the terminology of \cite{cms}, is equivalent to the germ of a tubular structure.
\end{example}


\subsection{Constructing frames}

In this subsection, we finally show our first main result, Theorem \ref{thm:1} below, stating that any simplicial manifold does admit a frame. This will be fundamental for the  subsequent study of geometric differentiation.


Fix $G$ a simplicial manifold. Recall that we regard the horn space as a topological subspace $\iota=\prod_{i>0} d_i:G_{n,0}\subset \prod_{i>0}G_{n-1}$, that its germ around $G_0$ is always that of an embedded submanifold, and that this is indeed an embedded submanifold if $G$ is a higher Lie groupoid. 
We call
$d_{n,0}(S_{n,j})\subset G_{n,0}$ the {\bf projected degeneracies}, which is a subspace of $G_{n,0}$:

$$\xymatrix{ & S_{n,j}\subset G_n \ar[rd]^{\prod_{i>0}d_i} \ar[d]^{d_{n,0}} & \\ G_{n-1}\ar[ru]^{s_{j-1}} \ar[r] & d_{n,0}(S_{n,j})\subset G_{n,0} \ar[r]^{\iota} & \prod_{i>0} G_{n-1}}$$

\begin{proposition}\label{prop:proj-deg}
Given $G$ a simplicial manifold, the projected degeneracies $d_{n,0}(S_{n,j})$ form an embedded submanifold of the product $\prod_{i>0}G_{n-1}$, and $d_{n,0}:S_{n,j}\to d_{n,0}(S_{n,j})$ is a diffeomorphism.
\end{proposition}

\begin{proof}
The composition $(\prod_{i>0}d_i)s_{j-1}:G_{n-1}\to\prod_{i>0}G_{n-1}$ is an embedding, for its $j$-th component is an identity by the simplicial identities, and therefore it admits a left inverse. Since $s_{j-1}:G_{n-1}\to G_n$ is also a closed embedding, so does the restriction $\prod_{i>0}d_i:S_{n,j} \to \prod_{i>0}G_{n-1}$. 
Note that $(g_1,\dots,g_n)\in \prod_{i>0}G_{n-1}$ is in $d_{n,0}(S_{n,j})=(\prod_{i>0}d_i)(S_{n,j})$ if and only if it satisfies the following equations:
$$g_i=\begin{cases} s_{j-2}d_i(g_j)  & i<j-1\\
   g_j & i=j-1\\
   s_{j-1}d_i(g_j) &  i>j\end{cases}$$ 
\end{proof}


\begin{remark}[Dold-Kan II]\label{rmk:DK2}
Continuing \ref{rmk:dk1}, the {\bf denormalization} is the inverse to $N$,
$$K:Ch_{\geq0}(VB(M))\to sVB(M)$$
it associates to a connective chain complex $(E,\partial)$ on $VB(M)$ a simplicial vector bundle $E_K=((E_K)_n,d_i,s_j)$. Each $(E_K)_n$ is usually described as a sum over surjective order maps $[n]\to[k]$. We use subsets $\alpha\subset\{1,\dots,n\}$ instead, in the spirit of \cite{cc}. Write $e^n_1,\dots, e^n_n$ for the canonical basis in $\Z^n$, and $e^n_\alpha=e^n_{\alpha_1}\wedge\dots\wedge e^n_{\alpha_k}\in \Lambda^k(\Z^n)$ for ordered indices $\alpha_i < \alpha_{i+1}$ in $\alpha=\{\alpha_1,\dots \alpha_k\}$. Then

\begin{enumerate}[a)]
    \item $(E_K)_n=\bigoplus_{k=0}^n E_k\otimes\Lambda^k(\Z^n)$.  The space $(E_K)_n$ is spanned by tensors $v\otimes e^n_\alpha$, $v\in E_k$, $|\alpha|=k$;
    \item $s_j:(E_K)_n\to (E_K)_{n+1}$ is given by $s_j(v\otimes e^n_\alpha)=v\otimes e^{n+1}_{\delta_{j+1}(\alpha)}$, where $\delta_{j+1}:[n]\to[n+1]$;
    \item $d_i:(E_K)_n\to (E_K)_{n-1}$, with 
    $$d_i(v\otimes e^n_\alpha)=\begin{cases}
        v\otimes e^{n-1}_{\sigma_i(\alpha)} & i\neq \alpha \text{ or } i+1 \neq \alpha
         \\
        \partial(v)\otimes e^{n-1}_{\sigma_0(\alpha)\setminus\{0\}} & \text{$i=0$ and $1\in \alpha$ }\\
        0 & \text{otherwise}
        \end{cases}$$
\end{enumerate}
Note that $|\sigma_i(\alpha)|=|\alpha|-1$ if and only if  $i,i+1\in\alpha$. 
There is a correspondence between surjective maps $\tilde \alpha:[n]\to[k]$ and subsets $\alpha\subset\{1,\dots,n\}$ given by  $\tilde\alpha(t)=|\{1,\dots,t\}\cap \alpha|$, under which the above formulas correspond to the usual ones in the literature, see e.g. \cite[1.2.3.5]{Lurie}. 
\end{remark}

Let $V$ be a simplicial vector bundle, write $E=V_N$ for its normalization, and identify $V=K(E)$. The degeneracies of $V$ are given by $S_{n,j}=s_{j-1}(E_{n-1})=\bigoplus_{j\notin \alpha,\ |\alpha|=k} E_k\otimes e^n_\alpha\subset\bigoplus_{\alpha, |\alpha|=k}E_k\otimes e^n_\alpha$. 
Thus, the obvious projections $\pi_j:V_n\to S_{n,j}\subset V_n$ commute, namely $\pi_{j'}\pi_j=\pi_{j}\pi_{j'}$ for every $j,j'$. A key step in the proof of the Theorem will rely on the following extension property for microbundle maps, which we explore independently, for the sake of clarity.


 %
\begin{lemma}\label{lemma:nice-family}
Let $V\to M$ be a vector bundle and $\pi_j:V\to S_{j}\subset V$, $j=1,\dots,n$, commuting projections onto subbundles. If  $g_j:(S_j,M,\pi)\to (W,N,r)$ are microbundle maps that agree on the intersections, then there is a microbundle map $g:(V,M,\pi)\to (W,N,r)$ extending them.
\end{lemma}

\begin{proof}
This follows from the inclusion-exclusion principle.
Using a tubular neighborhood, we can assume the microbundle $(W,N,r)$ is given by a vector bundle. For each nonempty subset $\alpha\subset \{1,\dots,n\}$, $|\alpha|=k$, write $S_\alpha=\bigcap_{j\in \alpha}S_j$, $g_\alpha=g_j|_{S_\alpha}$ where $j\in \alpha$ is arbitrary, and $\pi_\alpha=\pi_{\alpha_1}\dots\pi_{\alpha_k}:V\to S_\alpha$.
Then we can define 
$$g:V\to W \qquad g(x)=\sum_\alpha (-1)^{|\alpha|+1}g_\alpha(\pi_\alpha(x))$$
The formula is clearly smooth, and it commutes with the microbundle projections. Note that if $x\in S_j$, then $\pi_j(x)=x$, and for each $\alpha$ not containing $j$ the terms corresponding to $\alpha$ and $\alpha\cup\{j\}$ cancel out, so only the term $\alpha=\{j\}$ survives, hence yielding $g(x)=g_j(x)$.
\end{proof}

We are now in condition to state and prove our first main result.

\begin{theorem}\label{thm:1}
Every simplicial manifold $G$ admits a frame $\phi$.
\end{theorem}

\begin{proof}
Write $\nu_n=\nu(G,G_0)_n$ and $\nu_{n,0}=\nu(G,G_0)_{n,0}$. We build the microbundle tubular neighborhoods $\phi_n:(\nu_n,G_0,e)\to (G_n,G_0,e)$ inductively on $n$. Note that $\nu_1=A_1$. We start by taking $\phi_1:(A_1,G_0,e) \to (G_1,G_0,e)$ as any microbundle tubular neighborhood. The rest of the proof consists of explaining the inductive step.

Suppose that the tubular neighborhoods $\phi_k$, $k<n$, are given in a way compatible with the positive faces and the degeneracies, and let us construct $\phi_n$. The idea is to regard $(G_n,G_0,e)$ as the central term of a  short exact sequence of microbundles, split it, and linearize factor-wise:
$$(d_{n,0}^{-1}(G_0),G_0,e) \hookrightarrow (G_n,G_0,e) \overset{d_{n,0}}{\to} (G_{n,0},G_0,e)$$ 
Writing the horn as the standard equalizer, since the $\phi_k$ commute with the face maps, the $\phi_k$, $k<n$, already induce a linearization $\phi_{n,0}$ of the horn $(G_{n,0},G_0)$. Indeed, 
$\prod_{i>0}\phi_{n-1}$ restricts to an isomorphism $\phi_{n,0}
: (\nu_{n,0},G_0,e)\to (G_{n,0},G_0,e)$:
$$\xymatrix{
\nu_{n,0} \ar[r]^(.4)\iota \ar[d]_{\phi_{n,0}} &  \prod_{i>0}\nu_{n-1} \ar[d]^{\prod_{i>0}\phi_{n-1}} \ar@<.5ex>[r] \ar@<-.5ex>[r] & \prod_{i>j>0}\nu_{n-2} 
\ar[d]^{\prod_{i>0}\phi_{n-2}}\\
G_{n,0} \ar[r]^(.4)\iota  & \prod_{i>0}G_{n-1} \ar@<.5ex>[r] \ar@<-.5ex>[r] & \prod_{i>j>0}G_{n-2} 
}$$

Write $S'_{n,j}=s_{j-1}(\nu_{n-1})\subset \nu_n$. We claim that $\phi_{n,0}$ restricts well to the projected degeneracies, in the sense that the left square below commutes. This follows by the commutativity of the large square, which can be checked at each component, combined with Proposition \ref{prop:proj-deg} and the inductive hypothesis. 
$$\xymatrix{
\nu_{n-1} \ar[rr]^{d_{n,0}s_{j-1}} \ar[d]_{\phi_{n-1}}& &  d_{n,0}(S'_{n,j})\subset\nu_{n,0}  \ar[d]_{\phi_{n,0}}  \ar[r]^\iota &  \prod_{i>0}\nu_{n-1} \ar[d]^{\prod_{i>0}\phi_{n-1}}
\\
G_{n-1} \ar[rr]^{d_{n,0}s_{j-1}} & & d_{n,0}(S_{n,j})\subset G_{n,0} \ar[r]^\iota   & \prod_{i>0}G_{n-1} 
}$$

By Dold-Kan, we have a canonical decomposition $\nu_n = \nu_{n,0} \oplus A_n$. 
The next step is to define $\phi_n$ lifting $\phi_{n,0}$ along $d_{n,0}:G_n \to G_{n,0}$, so that $D\phi_n|_{A_n}=\id_{A_n}$:
$$\xymatrix{
A_n \ar[r] & A_n\oplus\nu_{n,0} \ar@{-->}[d]^{\phi_n} \ar[r]^{d_{n,0}} & \nu_{n,0} \ar[d]^{\phi_{n,0}}\\ & 
 G_n \ar[r]^{d_{n,0}} &  G_{n,0}  }$$
The inverse of such a $\phi_n$ is of the form $(r,\phi_{n,0}^{-1}d_{n,0})$, where $r$ is a microbundle map
$$
r: (G_n,G_0,e) \to (A_n,G_0,e) \text{ such that }Dr|_{A_n\subset TG_n|_{G_0}}= \id_{A_n} 
$$
The resulting $\phi_n$ will automatically be compatible with the positive faces, for it lifts $\phi_{n,0}$, and $\phi_n$ will be compatible with the degeneracies if and only if $r(S_{n,j})=0$. In fact, $\phi_{n,0}$ restricts well to the projected degeneracies, and by the Dold-Kan formulas \ref{rmk:DK2}, $S'_{n,j}\subset \nu_{n,0}=\bigoplus_{k<n}E_k\otimes\Lambda^k(\Z^n)$. The last step in the proof is to show that we can choose such an $r$.

A map $r: (G_n,G_0,e) \to (A_n,G_0,e)$ such that $Dr|_{A_n\subset TG_n|_{G_0}}= \id_{A_n}$ can be constructed, for example, using a tubular neighborhood $(G_n,G_0)\simeq (\nu_n, G_0)$ and then projecting onto $A_n$. There is also a gauge-like freedom in choosing $r$: given a microbundle map 
$\sigma:(\nu_{n,0},G_0,e) \to (A_n,G_0,e)$
we can consider the modification of $r$ by $\sigma$ as 
$$ r_\sigma: G_n \to A_n, \ r_\sigma(x)= r(x)-\sigma(\phi_{n,0}^{-1}d_{n,0}(x)).$$
It is easy to see that $r_\sigma$ also satisfies $Dr_\sigma|_{A_n\subset TG_n|_{G_0}}= \id_{A_n}$, since $A_n\subset Ker(Td_{n,0})$.
Starting with an arbitrary $r$, we will modify it via a gauge transformation $\sigma$ as follows. The maps 
$$d_{n,0}(S'_{n,j})\subset\nu_{n,0} \to A_n \qquad x\mapsto r\rho_j\phi_{n,0}(x)$$
agree in the intersections, and the projections $\pi_j:\nu_{n,0}\to d_{n,0}(S'_{n,j})$ commute, so we can apply Lemma \ref{lemma:nice-family} to get a microbundle extension $\sigma: (\nu_{n,0},G_0,e) \to (A_n,G_0,e)$. It is straightforward to check that $r_\sigma(s_{j-1}(x))=r(s_{j-1}(x))-\sigma(\phi_{n,0}^{-1}d_{n,0}(s_{j-1}(x)))=0$, completing the proof.
\end{proof}


In the next section, we use frames to compare smooth functions on simplicial manifolds with polynomial functions on the tangent complex, and to define the differentiation functor from simplicial manifolds to higher Lie algebroids.


\section{The differentiation functor}

In this section, we explore the geometry of a simplicial manifold $G$ around $G_0$, and prove Theorem \ref{thm:2}, which both defines and characterizes the higher Lie algebroid $A_G$ of $G$. Concretely, the Chevalley-Eilenberg algebra $CE(A_G)$ is defined to be the quotient $C_N(G)/\J$, where $C_N(G)$ are the normalized cochains on $G$ and $\J$ is generated by cochains with higher vanishing order on the degeneracies. We also present a variant of the result with forms instead of functions.


\subsection{The differentiating ideal} \label{subsec:tangent-cochains}

In this subsection, we review the notion of cochains and normalized cochains on a simplicial manifold, and introduce the key differentiating ideal $\hat J\subset C_N(G)$.


Let $G$ be a simplicial manifold, $S_{n,j}=s_{j-1}(G_{n-1})\subset G_n$ the $j$-th submanifold of degenerate simplices, and $S_n=\bigcup_{j=1}^n S_{n,j}\subset G_n$ the degeneracy locus, which is a union of embedded submanifolds. It follows from the Frame Theorem \ref{thm:1} that these submanifolds intersect cleanly, at least near $G_0$. An {\bf $n$-cochain} is a function $f\in C^\infty(G_n)$, and it is {\bf normalized} if it vanishes over $S_n$.
We write $C^\infty_N(G_n)\subset C^\infty(G_n)$ for the ideal of normalized functions.


Given $f\in C^\infty(G_p)$ and $g\in C^\infty(G_q)$, the cup product $f\cup g\in C^\infty(G_{p+q})$ is given by
$(f\cup g)(x)=f((d_{p+1})^q(x)) . g((d_0)^p(x))$,  where $x\in G_{p+q}$ is a $(p+q)$-simplex, $(d_{p+1})^q(x)\in G_p$ is its front $p$-face, and $(d_0)^p(x)\in G_q$ is its back $q$-face. The differential $\delta(f)\in C^\infty(G_{p+1})$ is given by 
$\delta(f)(x)=\sum_{i=0}^{n+1} (-1)^i f(d_i(x))$. 
The Leibniz rule relating $\cup$ and $\delta$ is
$$\delta(f\cup g)=\delta(f)\cup g + (-1)^p f\cup \delta(g) \qquad f\in C^\infty(G_p),\ g\in C^\infty(G_q)$$
Normalized cochains are closed under the cup product and the differential.

\begin{definition}
The {\bf algebra of cochains} is $C(G)=\bigoplus_n C^\infty(G_n)$ seen as a differential graded algebra endowed with the cup product and the differential. The {\bf algebra of normalized cochains} $C_N(G)$ is the sub-dga consisting of the normalized cochains.    
\end{definition}


We introduce now our differentiating ideal which is generated by the cochains with higher vanishing at the degeneracy locus, as follows. Write $I(S_{n,j})=\{f:f|_{S_{n,j}}=0\}\subset C^\infty(G_n)$, so that $C^\infty_N(G_n)=\bigcap_{j=1}^n I(S_{n,j})$, 
and $f\in I(S_{n,j})^2$ if it vanishes quadratically at $S_{n,j}$. 

\begin{definition} \label{def:diffideal}
For $n>0$, the {\bf higher vanishing ideal} $J_n\subset C_N^\infty(G_n)$ is spanned by the functions with higher vanishing in at least one of the degenerate submanifolds $S_{n,j}$. The {\bf differentiating ideal} $\J\subset C_N(G)$ is the differential ideal generated by $\Jgen=\bigoplus_n \Jgen_n$.
$$\Jgen_n=
\sum_{j=1}^n(C^\infty_N(G_n)\cap I(S_{n,j})^2)\subset C_N^\infty(G_n) \qquad
J=\bigoplus_n \Jgen_n  \subset C_N(G) \qquad 
\hat J = J+\delta(J) \subset C_N(G) 
$$
\end{definition}

By the Leibniz rule, it is easy to check that both $J$ and $J+\delta(J)$ are two-sided ideals for the cup product, and that is why $\hat J=J+\delta(J)$ is the differential ideal spanned by $J$.

\begin{remark}
Write $\ord(f,S_i)$ for the order of vanishing of $f:G_n\to \R$ at $S_i\subset G_n$, and $\ord(f,S)=\sum_{i=1}^n \ord(f,S_{n,i})$ for the {\bf total order}. We can characterize $J$ as the ideal spanned by the normalized cochains of total order $\geq n+1$.
In the terminology of \cite{LoMe23}, $\ord$ is the \emph{total weighting} induced by the multi-weighting defined by the submanifolds $S_{n,j}$.   
\end{remark}

\begin{example}
The $\cup$-ideal $\Jgen$ is not differential in general, so that $\J\neq \Jgen$, as we can see in the next simple example. Consider $G$ the nerve of the abelian Lie group $\R\toto\ast$, so $G_2\cong \R^2$, $S_{2,j}=\{x:x_j=0\}$ are the coordinate axis, and $d_0,d_1,d_2:\R^2\to\R$ are given by $d_0(x)=x_1$, $d_1(x)=x_1+x_2$, and $d_2(x)=x_2$. The 1-cochain $f:\R\to \R$, $f(x)=x^2$, is in $\Jgen_1=\J_1\subset C_N^1(G)$, for it vanishes quadratically on $S_{1,1}=\{0\}\subset\R$, but $\delta(f):\R^2\to\R$, $\delta(f)(x_1,x_2)=x_1^2+x_2^2-(x_1+x_2)^2=-2x_1x_2$ is not in $\Jgen_2$. 
\end{example}


Since we are aiming at interpreting $C_N^\infty(G)/\J$ as a Chevalley-Eilenberg algebra, we need to further understand $J$. Our first observation is that quotienting by $J$ localizes around the units $G_0$.

\begin{lemma} \label{lem:Jgenlocal}
If $f,g\in C^\infty_N(G_n)$ agree on an open $U$ around $G_0\subset G_n$ then $f-g\in \Jgen_n$.
\end{lemma}

\begin{proof}
Suppose that $f\in C_N^\infty(G_n)$, $f|_U=0$, and let us show that $f\in J_n$.
Let $d$ be a distance for the topology in $G_n$.
For each $j$ we consider the opens 
$$U_j=\{x: d(x,S_{n,j})<\max_{k\neq j}d(x,S_{n,k})\} \qquad \text{and} \qquad V_j=U_j\cup U.$$
We claim that $S_{n,j}\subset V_j$. If 
$x\in \cap_k S_{n,k}=G_0$ then $x\in U$, and if
$x\in S_{n,j}$ is not in some $S_{n,k}$ then $0=d(x,S_{n,j})<\max_{k\neq j}d(x,S_{n,k})$ and $x\in U_j$.
Note that $\cap_j V_j=U$, for if $x\in \cap_j V_j$ and $k$ maximizes $d(x,S_{n,k})$, then $x\in V_k\setminus U_k\subset U$.
Let $\lambda_j:G_n\to[0,1]$ be a smooth map such that $\lambda_j=1$ in a neighborhood of $S_{n,j}$ and $\lambda_j=0$ outside $V_j$. Finally,
$$f=f\prod_j (\lambda_j + (1-\lambda_j))\in J_n$$
because $f\prod_j\lambda_j=0$ and $f(1-\lambda_j)\in C^\infty_N(G_n)\cap I_j^2$ for each $j$.
\end{proof}

\begin{corollary}\label{cor:framemodJ}
A frame $\phi$ of a simplicial manifold $G$ yields linear isomorphisms between the quotients $\phi_n^*:C_N^\infty(G_n)/\Jgen_n\to C_N^\infty(\nu(G_n,G_0))/\Jgen_n$ for every $n\geq0$.
\end{corollary}


\subsection{Normal coordinates and monomials}


In this subsection, we study cochains modulo the differentiating ideal using normal coordinates and monomials. We show that the quotients $C_N^\infty(G_n)/J_n$ and $C_N^\infty(G_n)/\hat J_n$ are finitely generated since these classes are completely determined by their Taylor expansion normal to the units $G_0$. We also give formulas for the induced cup product and the differential on these quotients.


Following Theorem \ref{thm:1} and Corollary \ref{cor:framemodJ}, and to ease the notations, we will assume that $G_n=\nu(G,G_0)_n\cong\bigoplus_{k=1}^n A_k\otimes\Lambda(\Z^n)$, and that the positive faces and the degeneracies are given as in Dold Kan \ref{rmk:DK2}. We write $\pi:G_n\to G_0$ for the projection.

\begin{definition}
Given $\alpha\subset\{1,\dots,n\}$, $|\alpha|=k$, and given $\ell\in\Gamma A_k^*$, $\ell\neq0$, the {\bf normal coordinate function} $x_{\alpha,\ell}: G_n \to \R$ is the corresponding fiberwise linear function, namely
$$x_{\alpha,\ell}(v\otimes e^n_\beta)=
 \begin{cases}
\ell_{\pi(v)}(v) & \alpha=\beta \\
             0 & \text{otherwise}.
                      \end{cases}$$
Given a labeled sequence $(P,L)= 
((\alpha_1,\ell_1),\dots,(\alpha_s,\ell_s))$, 
the {\bf normal monomial} is the product of the corresponding normal coordinate functions:
$$x_{P,L}=x_{\alpha_1,\ell_1}\dots x_{\alpha_s,\ell_s}.$$
\end{definition}


The following is a direct consequence of the definition and of the Dold-Kan formulas.
\begin{lemma}\label{lemma:faces-degeneracies}
The positive faces and degeneracies are given in normal coordinate functions by the following formulas: 
$$x_{\alpha,\ell}\circ s_{j-1}= \begin{cases}
    x_{\sigma_{j-1}(\alpha),\ell} & j\notin \alpha \\
    0 & j\in \alpha \end{cases}
    \qquad
x_{\alpha,\ell}\circ d_i= \begin{cases}
    x_{\delta_i(\alpha),\ell} & i\notin \alpha \\
    x_{\delta_i(\alpha),\ell}+x_{\delta_{i+1}(\alpha),\ell} & i\in \alpha \end{cases}$$
\end{lemma}


We then see that the function $x_{\alpha,\ell}$, $|\alpha|=k$, vanishes over $S_{n,j}=s_{j-1}(G_{n-1})$ if and only if $j\in \alpha$, and in that case $\ord(x_{\alpha,\ell},S_{n,j})=1$. Thus, the total order of $x_{\alpha,\ell}$ is exactly $k$. 
Note that $\delta_{i+1}(\alpha)=\delta_i(\alpha)$ if $i\notin\alpha$, $0<i\neq n$, and that $n\notin\alpha\subset\{1,\dots,n-1\}$. 
Given an ordered sequence $P=(\alpha_1,\dots,\alpha_s)$, we say that it is a {\bf covering} if $\bigcup_{i=1}^s\alpha_i=\{1,\dots,n\}$, and that it has {\bf overlap at $k$} if $k\in \alpha_i\cap\alpha_j$ with $i\neq j$. A {\bf partition} $P$ is a covering without overlaps. With these, the normal monomial $x_{P,L}$ is a normalized cochain if and only if $P$ is a covering, and $x_{P,L}$ is in $J$ if and only if $P$ is a covering with overlap. 


\begin{proposition}\label{prop:generators}
\begin{enumerate}[a)]
    \item $I(S_{n,j})$ is spanned by the normal coordinate functions $x_{\alpha,\ell}$ with $j\in\alpha$. 
    In other words, if $f\in I(S_{n,j})$ then it admits a finite expansion as follows:
$$f=\sum_{j\in\alpha}\sum_{\ell } f_{\alpha,\ell} x_{\alpha,\ell}$$
Moreover, if $f\in I(S_{n,j})^2$ then we can assume that each $f_{\alpha,\ell}$ is in $I(S_{n,j})$ too, and if $f\in I(S_{n,j'})$ and $j'\notin\alpha$, then we can assume that $f_{\alpha,\ell}\in I(S_{n,j'})$.
    \item  $C_N(G_n)$ is spanned by the normal monomials $x_{P,L}$ with $P$ a covering. 
    %
    \item 
    $J_n$ is spanned by the normal monomials $x_{P,L}$ with $P$ a covering with at least one overlap.
\end{enumerate}
\end{proposition}

\begin{proof}
Suppose first that $f$ has a small support, in the sense that it is contained in an open $U$ trivializing the $A_k$, $n\geq k$, and take for each $k$ a finite family $B_k\subset\Gamma(A_k^*)$ giving a basis of sections for $A_k^*|_{U}$.
Given $x\in G_n$, consider the curve $t\mapsto x^t$, with $\pi_\alpha(x^t)=t\pi_\alpha(x)$ if $j\in\alpha$ and $\pi_\alpha(x^t)=\pi_\alpha(x)$ if $j\notin\alpha$. Then $x^1=x$, $x^0\in S_{n,j}$, and 
$$f(x)=f(x^1)-f(x^0)=\int_0^1 \frac{d}{dt}f(x^t)dt 
= \sum_{j\in\alpha}\sum_{\ell\in B_{|\alpha|}} x_{\alpha,\ell}f_{\alpha,\ell}(x)$$
where $f_{\alpha,\ell}(x)=\int_0^1 \frac{\partial}{\partial x_{\alpha,\ell}}f(x^t) dt$. These formulas imply that
 $\ord (f_{\alpha,\ell},S_{n,j})=\ord (f,S_{n,j})-1$ and that
$\ord (f_{\alpha,\ell},S_{n,j'})=\ord (f,S_{n,j'})$ for $j'\notin\alpha$, proving a).
If $f$ is of small support and either normalized or with higher vanishing, we can iterate the above expansions and prove b) and c).
%
Finally, if $f$ has no small support, then we can cover  $G_0$ with finitely many small opens $U_i$, in the sense that they trivialize each of the vector bundles $A_k$, $n\geq k$, and use a partition of 1 to write $f=\sum_i f_i$ with $\supp f_i\subset U_i$ and $\ord(f_i,S_{j'})\geq \ord(f,S_{j'})$ for every $j'$.
\end{proof}

We say that a partition $P=(\alpha_1,\dots,\alpha_s)$ is {\bf well-ordered} if $\min \alpha_i<\min\alpha_{i+1}$ for every $i$. Of course, any partition can be rearranged into a unique well-ordered one.

\begin{corollary}
The $n$th order Taylor expansion along the $\pi$-fibers, $T_{\leq n}:C^\infty(G_n)\to \Gamma(G_0,S^\bullet(G_n^*))$, yields a $C^\infty(G_0)$-isomorphism 
$$C^\infty_N(G_n)/\Jgen_n \cong \bigoplus_{\substack{P=(\alpha_1,\dots,\alpha_s) \\ \text{ well-ordered partition}}}
\Gamma(A_{{k_1}}^*)\otimes \dots \otimes \Gamma(A_{{k_s}}^* ) \qquad k_i=|\alpha_i|$$
\end{corollary}

\begin{proof}
As in the proof of Proposition \ref{prop:generators}, suppose first that $f\in C_N^\infty(G_n)$ has a small support contained in $U$, for each $k$ we fix a finite family $B_k\subset\Gamma(A_k^*)$ giving a basis for $A_k^*|_{U}$, and we can restrict the labelings to these families calling them \emph{basic}. Then $$f\equiv_{J_n} \sum_{P \text{ partition}}\left(\sum_{L \text{ basic labeling}}  f_{P,L}x_{P,L}\right)$$
where, if $(P,L)=((\alpha_1,\ell_1),\dots,(\alpha_s,\ell_s))$, then $f_{P,L}$ is uniquely determined as the $\pi$-basic function corresponding to $(\frac{\partial}{\partial x_{\alpha_1,\ell_1}}\cdots\frac{\partial}{\partial x_{\alpha_s,\ell_s}}f)|_{G_0}$. For a general $f$ we use a partition of unit argument.
\end{proof}

Given a sequence $P=(\alpha_1,\dots,\alpha_s)$, we write  $|P|=\bigcup_i \alpha_i \subset\{1,\dots,n\}$. Thus, $P$ is a covering if and only if $|P|=\{1,\dots,n\}$. 

\begin{lemma}\label{lemma:local-d0}
If $j\in \delta_0(\alpha)\cup\{1\}$ then $x_{\alpha,\ell}d_0-x_{\delta_0(\alpha),\ell}$ vanishes over $S_{n,j}$, and we can write 
$$x_{\alpha,\ell} \circ d_0= x_{\delta_0(\alpha),\ell} + \sum_{|P|\supset \delta_0(\alpha)\cup\{1\}}\sum_{L}   f_{P,L}x_{P,L}$$
\end{lemma}

\begin{proof} 
For $j=1$, we have $x_{\alpha,\ell}d_0s_0=x_{\alpha,\ell}$, and by Lemma \ref{lemma:faces-degeneracies}, $x_{\delta_0(\alpha),\ell}s_0=x_{\alpha,\ell}$, thus
$x_{\alpha,\ell}d_0$ and $x_{\delta_0(\alpha),\ell}$ agree over $S_{n,1}=s_0(G_{n-1})$. 
Suppose now that $j\in \delta_0(\alpha)$. Then $x_{\delta_0(\alpha),\ell}$ vanishes over $S_{n,j}$. And since $j-1\in \alpha$, we have $x_{\alpha,\ell}s_{j-2}=0$, then $x_{\alpha,\ell}d_0s_{j-1}=x_{\alpha,\ell}s_{j-2}d_0=0$, and $x_{\alpha,\ell}d_0$ also vanishes over $S_{n,j}$. 
The last claim follows from Proposition \ref{prop:generators}.  
\end{proof}


\begin{corollary}\label{cor:deltaord}
$\ord (x_{\alpha,\ell}\circ d_i)\geq \ord(x_{\alpha,\ell})$ for all $i$, and $\ord(\delta(f))\geq \ord(f)$ for all $f\in C^\infty(G_n)$.
\end{corollary}

Given a sequence $P=(\alpha_1,\dots,\alpha_s)$,  $|\alpha_i|=k_i$, and given $\beta:[n]\to[m]$ we write 
$\beta(P)=(\beta(\alpha_1),\dots,\beta(\alpha_s))$.
Given $(P,L)=((\alpha_1,\ell_1),\dots,(\alpha_s,\ell_s))$ a labeled sequence and $\beta$ injective on each $\alpha_t$, we denote by $(\beta(P),L)=((\beta(\alpha_1),\ell_1),\dots,(\beta(\alpha_s),\ell_s))$ the sequence $\beta(P)$ with the induced labeling.



\begin{proposition}\label{prop:cup-product}
Given $x_{P,L}\in C_N^\infty(G_p)$ and $x_{P',L'}\in C_N^\infty(G_q)$ generating monomials, where $(P,L)$ and $(P',L')$ are labeled partitions, their cup product mod $J$ is given by 
$$x_{P,L} \cup  x_{P',L'}  \equiv_J
x_{(\delta_{p+1})^q(P),L}x_{(\delta_{0})^p(P'),L'}$$
\end{proposition}

\begin{proof}
By definition of the cup product, we have
$x_{P,L} \cup  x_{P',L'}=
(x_{P,L}(d_{p+1})^q) (x_{P',L'}(d_{0})^p)$.
By Lemma \ref{lemma:faces-degeneracies}, we have 
$x_{P,L}(d_{p+1})^q=x_{(\delta_{p+1})^q(P),L}$.
By iterations of Lemma \ref{lemma:local-d0}, we have
$$x_{P',L'}(d_0)^p= 
x_{(\delta_{0})^p(P'),L'}
+ \sum_{|\tilde P|\supset |(\delta_{0})^p(P')|\cup\{1\}}\sum_{\tilde L\in B_{\tilde P}}   f_{\tilde P,\tilde L}x_{\tilde P,\tilde L}$$
Since $|(\delta_{p+1})^q(P)|=\{1,\dots,p\}$, $|(\delta_{0})^p(P')|\cup\{1\}=\{1,p+1,\dots,p+q\}$, 
and the monomials with overlap belong to $J$, 
we conclude
$(x_{P,L}(d_{p+1})^q)(x_{P',L'}(d_0)^p)  \equiv_J
(x_{(\delta_{p+1})^q(P),L})(x_{(\delta_0)^p(P'),L'})$.
\end{proof}

Next we describe generators for $\delta(J)$. 
Recall that $J_{n-1}$ is generated by the monomials $x_{Q,L}=x_{\alpha_1,\ell_1}\dots x_{\alpha_s,\ell_s}$ of order $n$ where $Q=(\alpha_1,\dots,\alpha_s)$ is a covering with exactly one overlap at some $j\in\{1,\dots,n-1\}$.
Given such a covering $Q$, $j\in \alpha_r $, $j\in \alpha_{r'} $, $r<r'$, note that $\delta_{j}(\alpha_i)=\delta_{j+1}(\alpha_i)$ for every $i\neq r,r'$, while $\delta_{j}(\alpha_r)\neq\delta_{j+1}(\alpha_r)$ and $\delta_{j}(\alpha_{r'})\neq\delta_{j+1}(\alpha_{r'})$. 
The {\bf primitive partitions} $P,P'$ of $Q$ are the partitions of $\{1,\dots,n\}$ which are defined as follows:
%
$$P=(\delta_{j}(\alpha_1),\dots, \delta_{j}(\alpha_r),\dots,\delta_{j+1}(\alpha_{r'}),\dots,\delta_{j}(\alpha_s))$$
$$P'=(\delta_{j}(\alpha_1),\dots, \delta_{j+1}(\alpha_r),\dots,\delta_{j}(\alpha_{r'}),\dots,\delta_{j}(\alpha_s))$$
A label $L$ for $Q$ induces labels for $P$ and $P'$, which, by an abuse of notation, we also denote by $L$.

The $Q$ determines the $j$ and the two primite partitions $P,P'$. Conversely, given $P=(\alpha_1,\dots,\alpha_s)$ a partition of $\{1,\dots,n\}$ and given $j$ such that $j,j+1$ do not belong to the same $\alpha_i$, we can recover $P'$ and $Q$. Writing $\tau_j$ for the transposition of $j$ and $j+1$, we have
$P'=\tau_j(P)=(\tau_j(\alpha_1),\dots,\tau_j(\alpha_s))$ and $Q=\sigma_j(P)$, and we say that $\tau_j$ is an {\bf allowed transposition} for $P$.  
If $L$ is a label for $P$ then it induces a label for both $\tau_j(P)$ and $Q$ in the obvious way.
Thus, $P$ and $P'$ are the primitive partitions of some $Q$ if and only if they are related by an allowed transposition $\tau_j$ for $P$,  $P'=\tau_j(P)$.


\begin{lemma}\label{lem:gendeltaJgen}
Let $x_{Q,L}=x_{\alpha_1,\ell_1}\dots x_{\alpha_s,\ell_s}$ be a generator of $J_{n-1}$ of order $n$ with overlap at $j\in\{1,\dots,n-1\}$. Then 
$$\delta(x_{Q,L})\equiv_{\Jgen_{n}}
(-1)^j(x_{P,L}+x_{\tau_j(P),L})$$
where $P$ and $\tau_j(P)$ are the primitives of $Q$. 
Therefore, $\hat J$ is spanned by $J$ and the sums $x_{P,L}+x_{\tau_j(P),L}$, where $P$ is a partition and $\tau_j$ an allowed transposition.
\end{lemma}

\begin{proof}
Let $r,r'$ be such that  $j\in \alpha_r$, $j\in \alpha_{r'}$, $r<r'$.
Since $x_{Q,L}$ is normalized, so is $\delta(x_{Q,L})$ and we can focus on the terms in the Taylor expansion corresponding to the partitions.
By definition, 
$\delta(x_{Q,L})=\sum_{i=0}^{n} (-1)^i x_{Q,L}\circ d_i$.
%
We study each term $x_{Q,L}\circ d_i$ by splitting in cases.
If $i<j$, since $d_is_j=s_{j-1}d_i$, using Lemma \ref{lemma:faces-degeneracies}, we have that $x_{\alpha_r,\ell_r}d_i$ and $x_{\alpha_{r'},\ell_{r'}}d_i$ vanish at $S_{n,j+1}$. It follows that $x_{Q,L}d_i$ has higher vanishing at $S_{n,j+1}$, and therefore, the Taylor expansion along the fibers of $x_{Q,L}d_i$ does not contribute to any partition. The case $i>j$ is similar, now using the identity $d_is_{j-1}=s_{j-1}d_{i-1}$, so $x_{\alpha_r,\ell_r}d_i$ and $x_{\alpha_{r'},\ell_{r'}}d_i$ vanish at $S_{n,j}$.
Finally, consider the case $i=j$. Again by Lemma \ref{lemma:faces-degeneracies} we have
$$x_{Q,L}\circ d_j=
(x_{\delta_j(\alpha_r),\ell_r}+x_{\delta_{j+1}(\alpha_r),\ell_r})
(x_{\delta_j(\alpha_{r'}),\ell_{r'}}+x_{\delta_{j+1}(\alpha_{r'}),\ell_{r'}})
\prod_{t\neq r,r'}x_{\delta_j(\alpha_t),\ell_t}$$
Out of the four terms, only the cross terms contribute to a partition, each of these cross terms corresponds to a primitive $P$ and $P'$, and the result follows.
\end{proof}


We know that the quotient $C_N^\infty(G_n)/J_n$ is spanned by the normal monomials $x_{P,L}$, with $(P,L)$ a labeled partition. 
Given $(P,L)=((\alpha_1,\ell_1),\dots,(\alpha_s,\ell_s))$ a labeled partition and $\theta$ a permutation of $\{1,\dots,s\}$, we write $(P^\theta,L^\theta)=((\alpha_{\theta(1)},\ell_{\theta(1)}),\dots,(\alpha_{\theta(s)},\ell_{\theta(s)}))$.
Clearly $x_{P,L}=x_{P^\theta,L^\theta}$, and by Lemma \ref{lem:gendeltaJgen}, $x_{P,L}\equiv_{\hat J_n}-x_{\tau_j(P),L}$. 
Thus, to get generators for $C_N(G_n)/\hat J_n$, it is enough to take one monomial on each class defined by the equivalence relation $\sim$ spanned by:
\begin{enumerate}[$\bullet$]
\item $(P,L)\sim(P^\theta,L^\theta)$, where $\theta$ is a permutation of $\{1,\dots,s\}$
\item $(P,L)\sim(\tau_j(P),L)$, where $\tau_j$ is an allowed transposition.
\end{enumerate}

Given an integer partition $\lambda=(k_1,\dots,k_s)
\vdash n$,
namely a non-decreasing sequence $k_1,\dots,k_s$ of positive integers whose sum is $n$, we define the {\bf canonical} partition $P_\lambda=(\alpha_1,\dots,\alpha_s)$ by 
$$\alpha_1=\{1,\dots,k_1\} \quad \alpha_2=\{k_1+1,\dots,k_1+k_2\} \quad \dots\quad  \alpha_s=\{n-k_s+1,\dots,n \}$$

\begin{proposition}\label{prop:canonical-partition}
Given $(P,L)=((\alpha_1,\ell_1),\dots,(\alpha_s,\ell_s))$ a labeled partition, there is a labeled canonical partition $(P_\lambda,L')$ such that $(P,L)\sim(P_\lambda,L')$. Moreover, $P_\lambda$ is unique.
\end{proposition}

\begin{proof}
Write $k_i=|\alpha_i|$.
The collection $\{k_1,\dots,k_s\}$ remains constant over an equivalence class. After reordering, these are the numbers defining the integer partition $\lambda=(k_1,\dots,k_s)\vdash n$. We now describe an algorithm to {\it unravel} any partition by a sequence of allowed transpositions to get a canonical one. 
First, we use a permutation $\theta$ to get $(P,L)$, $P=(\alpha_1,\dots,\alpha_s)$, such for each $t$ either 
(i) $|\alpha_t|<|\alpha_{t+1}|$, or 
(ii)$|\alpha_t|=|\alpha_{t+1}|$ and 
$\max \alpha_t <\max \alpha_{t+1}$. 
Then, at each step of the algorithm, we seek the minimum $j$ such that $j\in \alpha_{r} $, $j+1\in \alpha_{r'} $, and $r'<r$. If there is no such $j$, then $P=P_\lambda$ is canonical, and we are done. If there is such $j$ then $\tau_j$ is an allowed transposition for $P$ and we replace $(P,L)$ by $(\tau_j(P),L)$,  completing one step of our algorithm. 
Each step reduces the number of disordered pairs
$$\#\{(i,j): i<j,\ i\in \alpha_r ,\ j\in \alpha_{r'} , \ r'<r \}$$
and therefore, after finitely many steps, we end up getting $(P_\lambda,L)$.
\end{proof}


\begin{corollary}\label{cor:generators}
The algebra $C_N^\infty(G)/{\hat J}$ is  graded commutative and spanned, as a graded $C^\infty(G_0)$-algebra, by the monomials $x_{\iota_n,\ell}$, where $n\geq 1$, $\iota_n=\{1,\dots,n\}$, and $\ell\in \Gamma(A_n^*)$.
\end{corollary}

\begin{proof}
By Propositions \ref{prop:generators}, the normal monomials $x_{P,L}$ span the quotient $C_N^\infty(G)/{ J}$, and by Proposition \ref{prop:canonical-partition}, the $x_{P_\lambda,L}$ with $\lambda=(k_1,\dots,k_s)\vdash n$ span the quotient $C_N^\infty(G)/{\hat J}$. Finally,
by Proposition \ref{prop:cup-product}, such a normal monomial can be written as a cup product, namely 
$$x_{P_\lambda,L}\equiv_J x_{\iota_{k_1},\ell_1}\cup\dots\cup x_{\iota_{k_s},\ell_s}.$$ 
To show that $C_N^\infty(G)/{\hat J}$ is graded commutative, it is enough to pick two generators 
$x_{\iota_{p},\ell}$, $x_{\iota_{q},\ell'}$, and show that
$x_{\iota_{p},\ell}\cup x_{\iota_{q},\ell'}\equiv(-1)^{pq}x_{\iota_{q},\ell'}\cup x_{\iota_{p},\ell}$, or equivalently, that $x_{\{1,\dots,p\},\ell}x_{\{p+1,\dots,p+q\},\ell'} \equiv 
(-1)^{pq} x_{\{1,\dots,q\},\ell'}x_{\{q+1,\dots,p+q\},\ell}$. 
We can go from the partition $(\{1,\dots,p\},\{p+1,\dots,p+q\})$ to $(\{q+1,\dots,p+q\},\{1,\dots,q\})$
through $pq$ allowed transpositions, and conclude by Lemma \ref{lem:gendeltaJgen}.
\end{proof}


\subsection{Higher Lie algebroids and the Differentiation Theorem}

In this subsection, we review the definition of a higher Lie algebroid and prove Theorem \ref{thm:2}, our main result, stating that the quotient $C_N(G)/\hat J$ is semi-free, and therefore the Chevalley-Eilenberg algebra of a thus defined higher Lie algebroid.


Given $C=(\bigoplus_{k\geq0}C^k,\wedge,d)$ a commutative differential graded algebra, we write $C^{\geq1}=\bigoplus_{k\geq1}C^k$ for the {\bf augmentation ideal}, and $Q(C)=C^{\geq1}/(C^{\geq1}\wedge C^{\geq1})$ for the {\bf module of indecomposables}, which can also be seen as the associated graded algebra for the filtration
$$F_kC=\{\text{subalgebra generated by $C^{\leq k}$}\}$$
Then $Q(C)^k$ are the degree $k$ indecomposables and the induced $d:Q(C)^k\to Q(C)^{k+1}$ is $C^0$-linear. A {\bf splitting} for $C$ is a family of $C^0$-linear sections $\iota_k:Q(C)^k\to C^k$ inducing a graded algebra isomorphism $\iota:S(Q(C))\xto\sim C$ from the symmetric algebra of the module of indecomposables. If such a splitting exists, then $C$ is called {\bf semi-free}.

\begin{definition}
Let $M$ be a manifold and $A=\oplus_{n\geq 1} A_n$ be a graded vector bundle over $M$. A {\bf higher Lie algebroid} $(A\to M,C)$ on $A$ is given by a semi-free differential graded algebra $C$ with $C^0=C^\infty(M)$ and whose module of indecomposables is $\Gamma(A^*)$, namely $Q(C)^k=\Gamma(A_k^*)$. In this case, $C=CE(A)$ is called the {\bf Chevalley-Eilenberg algebra} of the higher Lie algebroid.
\end{definition}

 By an abuse of notation, we often write just $A$ for the higher Lie algebroid. Such an $A$ is a {\bf Lie $m$-algebroid} if $A_k=0$ for $k>m$. A morphism of higher Lie algebroids $(A\to M,C)\to (A'\to M',C')$ consists of a vector bundle map $A\to A'$ together with an extension to a dga morphism $C'\to C$.

The {\bf tangent complex} $(A,\partial)$ of $(A\to M,C)$ is the chain complex of vector bundles defined by dualizing the differential on the indecomposables, $\partial:A_k\to A_{k-1}$, $\langle \partial(x),[f]\rangle=\langle x,[d(f)]\rangle$, 
$x\in\Gamma(A_k)$, $f\in \Gamma(A_{k-1}^*)=Q(C)^{k-1}$. 
The {\bf anchor} $\partial:A_1\to TM$ is a canonical augmentation given by $L_{\partial (x)}(f) =\langle x,d(f)\rangle$, $x\in \Gamma(A_1)$, $f\in C^\infty(M)$.

\begin{example}
Lie 0-algebroids are the same as manifolds, and Lie 1-algebroids are ordinary Lie algebroids, as defined classically. An important example of Lie 2-algebroid is that of a Courant algebroid $E\to M$, whose (un-augmented) tangent complex is $E^*\to E$, with the map given by the inner product on $E$. Higher Lie algebroids also appear, for instance, in \cite{holonomy} as resolutions of certain singular foliations.
\end{example}

\begin{remark}
By fixing a splitting, we can (non-canonically) encode $C=CE(A)$ as a degree 1 differential $d$ in $S(\Gamma A^*)$. Some authors define a higher Lie algebroid structure on $A$ as a dga $C$ that is isomorphic to $S(\Gamma(A^*))$. We are further demanding that the isomorphism is the identity on the indecomposables. By Serre-Swan, a semi-free differential graded $C^\infty(M)$-algebra $C$ is a Chevalley-Eilenberg algebra if each $Q(C)^k$ is finitely generated and projective.
\end{remark}

\begin{remark}[Supergeometric description]\label{rmk:superA}
A Higher Lie algebroid $A$ can be equivalently described as $\N$-graded supermanifold $\M_A$, such that $C^\infty(\M_A)=CE(A)$, and together with a homological vector field corresponding to the differential on $CE(A)$ (see \cite{JMP} and references therein). From this viewpoint, the existence of a splitting isomorphism is related to the existence of an atlas of graded coordinates.
\end{remark}


We are now in a position to state our differentiation theorem.

\begin{theorem}[Differentiation]\label{thm:2}
Let $G$ be a simplicial manifold, $A_G$ its tangent complex, and $\J\subset C_N(G)$ its differentiating ideal. Then, $C_N(G)/\J$ is a semi-free commutative differential algebra with indecomposables $\Gamma(A^*_G)$, hence uniquely defining a higher Lie algebroid on $A_G$ with Chevalley-Eilenberg algebra $CE(A_G)=C_N(G)/\hat J$. A frame $\phi$ for $G$ induces a splitting for $A_G$:
$$\phi_\star:S(\Gamma A_G^*)\to 
CE(A_G) \qquad \ell 
\in\Gamma A_n^*\mapsto [x_{\iota_n,\ell}\phi_n^{-1}] \in C(G_n)/\J_n$$
\end{theorem}

\begin{proof}
The quotient $C_N(G)/\J$ is a differential graded algebra, with $(C_N(G)/\J)_0=C^\infty(G_0)$, 
and by Corollary \ref{cor:generators} it is graded commutative, so we can define 
$\phi_\star:S(\Gamma A^*)\to 
CE(A_G)$ on the generators $\ell\in\Gamma A_n^*$ by $\phi_\star(\ell)=[x_{\iota_n,\ell}\phi^{-1}_n]$, and extend it to the whole symmetric algebra. 
This $\phi_\star$ is an epimorphism by Corollary \ref{cor:generators}.
It remains to prove that $\phi_\star$ is a monomorphism.

By Lemma \ref{lem:gendeltaJgen}, it is enough to show that the relation $x_{P,L}\equiv_{\hat J} - x_{\tau_j(P),L}$ induced by the allowed transpositions actually follows from graded commutativity. 
By Proposition \ref{prop:canonical-partition}, we can restrict to the case in which $P=P_\lambda$, $\lambda=(k_1,\dots,k_n)\vdash n$, and $(P,L)\sim(P,L^\theta)$, with $\theta$ a permutation of $\{1,\dots,s\}$. We want to compare the corresponding monomials $x_{P,L}$ and $x_{P,L^\theta}$. 
First, observe that $(\tau_j(P))^\theta=\tau_j(P^\theta)$. Thus, we can reach $(P,L^\theta)$ from $(P,L)$ by first performing a sequence of allowed transpositions and then a permutation, which is necessarily $\theta$: 
$$(P,L)\sim ((\tau_{i_1}\tau_{i_2}\dots\tau_{i_r}(P))^\theta,L^\theta)=(P,L^\theta)$$
If the $k_t$ are all distinct, then $\theta=\id$, for the ordering of the labeling is determined by that of $P$, and therefore, the product  $\tau_{i_1}\tau_{i_2}\dots\tau_{i_r}$ is also trivial, $r$ is even, and the signs coming from the allowed transposition cancel out. Suppose then that $k_t=k_{t+1}$ for some $t$, and suppose further that $\theta=\tau_t$ is the permutation of $\{1,\dots,s\}$ that swaps $t$ and $t+1$, 
the general case will then follow by writing any $\theta$ as a product of these transpositions.
The identity 
$\tau_{i_1}\tau_{i_2}\dots\tau_{i_r}(P)=P^{\tau_t}$ implies that the product $\tau_{i_1}\tau_{i_2}\dots\tau_{i_r}$ is the permutation in $\{1,\dots,n\}$ that swaps two consecutive blocks of size $k_t$, then its sign is $(-1)^{k_t^2}$, and the sign coming from the allowed transpositions yields
$x_{P,L}\equiv_{\hat J}(-1)^{k_t^2}x_{P,L^{\tau_t}}$. By Proposition \ref{prop:cup-product} this reads as 
$$x_{P_0,L_0}\cup
x_{\iota_{k_{t}},\ell_{t}}\cup x_{\iota_{k_{t+1}},\ell_{t+1}}\cup
x_{P_1,L_1}
\equiv
(-1)^{k_t}
x_{P_0,L_0}\cup
x_{\iota_{k_{t+1}},\ell_{t+1}}\cup x_{\iota_{k_{t}},\ell_{t}}\cup 
x_{P_1,L_1}
$$
where $(P_0,L_0)$ and $(P_1,L_1)$ come from the head and the tail of $P$ and $L$.
Since this relation indeed follows by graded commutativity, we conclude that $\phi_\star$ is injective and the result follows.
\end{proof}


Following Remark \ref{rmk:superA}, Higher Lie algebroids are also known as {\bf differential graded manifolds} and we write $dgMan$ for the corresponding category. Since our construction is clearly functorial, based on the previous theorem, we introduce the following definitions:

\begin{definition}
The {\bf higher Lie functor} associates to a simplicial manifold $G$ the higher Lie algebroid $(A_G,CE(A_G))$ with $CE(A_G)=C_N(G)/\hat J$.
$$\xymatrix{
   \text{sMan} \ar[rr]^{\text{Lie}} 
  & & \text{dgMan}
  }$$
At the level of cochains, we have an induced {\bf van Est map} defined as the quotient by $\J$, which is a natural epimorphism of differential graded algebras.
    $$\ve:C_N(G)\to C_N(G)/\J = CE(A_G)$$ 
\end{definition}
Notice that the tangent complex is preserved by the higher Lie functor and that, if $G$ is a Lie $m$-groupoid, then $A_G$ is a Lie $m$-algebroid. 

\begin{remark}[Germ differentiation]\label{rmk:germ-diff}
The functor $G\mapsto A_G$ only depends on the germ of $G$ around $G_0$. In fact, at each degre $n$, the differential $d:CE^n(A_G)\to CE^{n+1}(A_G)$ only depends on the $(n+1)$-jet of normalized functions and of the structure maps on $G_n$. Thus, we can define the differentiation of a simplicial germ $(G,G_0)$ around $G_0$.  
\end{remark}


The construction $G \mapsto A_G$ is invariant and does not depend on auxiliary choices. 
Given $v\in CE(A_G)$, we can compute $d_{CE(A_G)}(v)$ by finding a representative $f\in C_N^\infty(G)$, $[f]_{\J} =v$, so $d_{CE(A_G)}(v)=[\delta f]_{\J}$. But fixing a frame, and working locally, we can derive useful explicit local formulas for $d_{CE(A_G)}$ as in the following Remark.

\begin{remark}[The differential on generators]
Let $G$ be a simplicial manifold and $\phi$ a frame. Working locally on $G_0$, we can assume that we have a basis $\{e_n^i\}$ for each $\Gamma A_n^*$, and use them for labeling partitions. To ease the notation, given $\alpha\subset\{1,\dots,n\}$, write $(\alpha,i)=(\alpha,e^i_k)$, $x_\alpha^i=x_{\alpha,i}$, and $x_n^i=x_{\iota_n,i}$. 
By Lemmas \ref{lemma:faces-degeneracies} and \ref{lemma:local-d0}, 
since $\delta_0(\iota_n)\cup\{1\}=\{1,\dots,n+1\}$, and since the terms coming from $d_{i}$, $i>0$, cancel out with the first term in the expansion of $d_0$, we have:
$$ d[x_n^j] = 
\sum_{\substack{P=(\alpha_1,\dots,\alpha_s)\\\text{partition} }}
\sum_{\substack{I=(i_1,\dots,i_s)\\\text{labeling}}}  \left.\frac{\partial(x_{n}^j d_0)}{\partial x_{\alpha_1}^{i_1} \cdots \partial x_{\alpha_s}^{i_s}}\right|_{G_0} \ [x_{\alpha_1}^{i_1} \cdots x_{\alpha_s}^{i_s}]$$
The terms for $s=1$ above gives $\partial^*(x_{n})=x_n\partial$, where $\partial: A_{n+1} \to A_n$ is the differential on the tangent complex. 
For $s\geq 2$, following Proposition \ref{prop:canonical-partition}, any 
labeled partition is equivalent to a unique $(P_\lambda,I)$, 
where $\lambda=(k_1,\dots,k_s)\vdash n+1$ is a nontrivial integer partition, $k_t\leq k_{t+1}$, $k_1+\dots+k_s=n+1$, and $I=(i_1,\dots,i_s)$ is {\bf ordered}, in the sense that $i_t\leq i_{t+1}$ whenever $k_t= k_{t+1}$.
Writing $d_\phi=\phi_\star^{-1} \circ d \circ \phi_\star$ for the corresponding differential in $S(\Gamma(A^*))$, and writing $e_\lambda^I=e_{k_1}^{i_1} \cdots e_{k_s}^{i_s}$, we get:
$$d_\phi(e_n^j)=
\partial^*(e_n^j)+
\sum_{\substack{\lambda\vdash n+1\\ \text{nontrivial}\\\text{partition}}}
\sum_{\substack{I \text{ ordered}\\\text{labeling}}}
c^{\lambda}_{I} \ e^I_\lambda$$
The structure coefficients $c^\lambda_I$ are given by
$$c^{\lambda}_{I}=  
\sum_{\substack{(P,L) \text{ labeled partition}\\ (P,L)\sim (P_\lambda,I)}} (-1)^{(P,L)} \ \left.\frac{\partial(x_{n}^j d_0)}{\partial x_{\alpha_1}^{j_1} \cdots \partial x_{\alpha_s}^{j_s}}\right|_{G_0} \in C^\infty(G_0),
$$
where $(P,L)=((\alpha_1,j_1),\dots,(\alpha_s,j_s))$, and $(-1)^{(P,L)}$ is given by the number of allowed transpositions needed to get from $(P,L)$ to $(P_\lambda,I)$. 
From this formula, following \cite{BonPon}, we can deduce formulas for the underlying $m$-ary brackets $[\cdots]_{m}:\Gamma (\oplus_\bullet A_\bullet)^m \to \Gamma (\oplus_\bullet A_\bullet)$ in terms of the coefficients $c^\lambda_I$ given above. The higher Jacobi identity for such brackets is a highly nontrivial consequence of the simplicial identities on $G$, via $d^2=0$ in the invariantly defined $CE(A_G)$.
\end{remark} 


We explain here how our differentiation recovers the classical Lie algebroid of a Lie groupoid, exhibiting the classical Lie bracket as a second–order Taylor coefficient. 

\begin{proposition}\label{prop:gdcomparison}
Let $H\toto M$ be a Lie groupoid and $G$ its nerve. Then, the  classical differentiation $A_H$ and our simplicial differentiation $A_G$ are naturally isomorphic as Lie algebroids. Under this isomorphism, our van Est map $\ve$ also coincides with the classical one.
\end{proposition}

\begin{proof}
Consider the map $C_N(G_1)\to \Gamma A_H^*$ given by taking normal derivative to $M$ inside $H$. It is a clear that it induces an isomorphism of graded algebras $\Gamma(\Lambda A_H^*)\cong\Gamma(\Lambda A_G^*)$ since both are generated in degrees zero and one. We need to show that it commutes with the underlying differentials $d_H$, $d_G$. Both constructions have the same augmented tangent complex, hence the same anchor, so it remains to show that $d_H=d_G$ in degree 1.
Fix a frame $\phi$ in $G$, which is the same as a tubular structure in $H$, in the sense of \cite{cms}. 
Working locally, we can choose a local basis $e_i$ of $A_G$. This yields local coordinates $e^i$ in $G_1=H$ around $G_0$, and $e^i_1,e^j_2$ in $G_2=H _s\times_s H$ around $G_0$. Write $X_i$ for the right-invariant vector fields corresponding to $e_i$, and $\omega^i$ for the dual basis of right-invariant 1-forms. 
By the right invariance, we have 
$Fl^t_{X_i}(Fl_{X_j}^s(x))=Fl_{X_i}^t(y)Fl_{X_j}^s(x)$,
where $y=Fl_{\rho X_j}^s(x)$, and therefore, $Fl^s_{X_j}(x)^{-1}=Fl^{-s}_{X_j}(y)$. At the same time, the map $d_0:G_2\to G_1$ is the division map $d_0(g_2,g_1)=g_2g_1^{-1}$, where \smash{$z\xfrom{g_2} y\xfrom{g_1} x$}. Then we have, with $y=Fl_{\rho X_j}^s(x),\ z=Fl_{\rho X_i}^t(x)$,
\begin{align*}
d_H(\omega^k)(X_i,X_j)(x) 
    &= -[X_i,X_j](e^k)(x) \\
    &= -\left.\frac{\partial^2}{\partial t\partial s}\right|_{0}(e^k Fl^t_{X_i}(Fl_{X_j}^s(x))- 
    e^k Fl_{X_j}^s(Fl_{X_i}^t(x)) \\
    &= -\left.\frac{\partial^2}{\partial t\partial s}\right|_{0}(e^k d_0(Fl_{X_i}^t(y),Fl_{X_j}^{-s}(y)) - 
    e^k d_0(Fl_{X_j}^s(z),Fl_{X_i}^{-t}(z)) 
        && \\
    & = \left.\frac{\partial^2}{\partial t\partial s}\right|_{0}(e^k d_0(Fl_{X_i}^t(x),Fl_{X_j}^{s}(x)) - 
    e^k Fl_{X_j}^s(x)Fl_{X_i}^{t}(x))\\
    & =\left.\frac{\partial^2}{\partial e^{i}_{1}\partial e^{j}_{2}}\right|_{x}(e^k d_0)-
\left.\frac{\partial^2}{\partial e^{i}_{2}\partial e^{j}_{1}}\right|_{x}(e^k d_0)\\
    & = d_G(e^k)(e_i,e_j)(x)
\end{align*}
The last statement about $\ve$ follows from it holding in degrees zero and one by construction, and by direct verification on elements of the form $f_1\cup \dots \cup f_m$ with $f_j\in C^\infty_N(H)$ using the definitions.
\end{proof}

This computation shows that the Lie algebroid bracket is extracted functorially as the antisymmetric second derivative of the division map 
$d_0$. While this mechanism is implicit in classical constructions via invariant vector fields, the simplicial viewpoint isolates the bracket directly at the level of second jets of the groupoid structure. 


\subsection{Differentiation at the level of forms}\label{subsec:forms}

In this subsection, we review the Bott-Shulman complex of a simplicial manifold, the Weil algebra of a higher Lie algebroid, and mimic the previous constructions to get an extension of our differentiation theorem for forms instead of cochains.


Given $G$ a simplicial manifold, its {\bf cosimplicial de Rham algebra} is the cosimplicial differential graded algebra $(\Omega(G),d_i^*,s_j^*)$, see \cite{behrend,bss}. This yields a bi-differential graded algebra
$\Omega(G) = (\Omega^q(G_p), \cup, \delta, d)$
with horizontal simplicial differential $\delta:\Omega^q(G_p)\to \Omega^q(G_{p+1})$, $\delta(\omega)=\sum_{i=0}^{p+1} (-1)^i d_i^*\omega$, and vertical de Rham differential $d:\Omega^q(G_p)\to \Omega^{q+1}(G_p)$, $d(\omega)=d_{dR}(\omega)$, 
and the cup product analogous to the one defined for cochains,
$\alpha\cup \beta = (-1)^{qp'} (d^*_{p+1})^{p'}(\alpha). (d^*_0)^p(\beta)$, $\alpha\in \Omega^q(G_p)$, $\beta\in \Omega^{q'}(G_{p'})$. 
$$\Omega(G) \quad = \quad \begin{matrix}  
\xymatrix@R=15pt@C=15pt{ 
   \vdots & \vdots  & \vdots  & \reflectbox{$\ddots$} \\
  \Omega^2(G_0) \ar[u]  \ar[r] & \Omega^2(G_1) \ar[u]  \ar[r] & \Omega^2(G_2) \ar[u]  \ar[r] & \dots \\
  \Omega^1(G_0) \ar[u]  \ar[r] & \Omega^1(G_1) \ar[u]^d  \ar[r]^\delta   & \Omega^1(G_2) \ar[u]  \ar[r] & \dots  \\
  \Omega^0(G_0) \ar[u]  \ar[r] & \Omega^0(G_1) \ar[u]  \ar[r]  & \Omega^0(G_2) \ar[u]  \ar[r] &  \dots  
} \end{matrix} $$

A form $\omega\in \Omega^q(G_p)$ is {\bf normalized} if its pullback to each $S_{p,j}$ vanishes, namely $s_{j-1}^*(\omega)=0$ for $j=1,\dots,p$. The normalized forms are closed under the cup product, the horizontal and the vertical differential, hence they define a sub-bi-dga $\Omega_N(G)$. 

\begin{example}
If $G$ is a 0-groupoid, namely a constant simplicial manifold, then this double complex is homotopic to the usual de Rham complex. If $G$ is a Lie groupoid, then this construction is explored in detail by several authors, including \cite{behrend,bss,ac11}, among others. 
\end{example}

\begin{remark}
Given a simplicial manifold $G$, its {\bf shifted tangent} $\T G$ is the simplicial algebroid obtained by applying the tangent functor degreewise. 
We can identify the cosimplicial de Rham algebra with the degreewise Chevalley-Eilenberg algebra, $\Omega(G)\cong CE(\T G)$. 
\end{remark}


Let us now review the infinitesimal version of the above construction. 
Given $A$ a higher Lie algebroid, write $C=CE(A)$ for its Chevalley-Eilenberg algebra, $Der(C)$ for the $C$-module of graded derivations, and $\Omega^1_C=Hom_C(Der(C),C)$. 

\begin{definition}
The {\bf Weil algebra} of $A$ is the differential bigraded algebra with total space
$$W(A)
= \bigoplus_{k \ge 0} Hom_C(\Lambda_C^k Der(C), C)
\cong
\Lambda_C(\Omega^1_C).
$$
The vertical grading is $k$ above, the vertical differential on elements $f \in C$ with $k=0$ is given by
$d_v(f)(X) = X(f)$, $X \in Der(C)$,
and $d_v$ extended as a $(0,1)$-graded derivation. The horizontal grading comes from the internal grading of $Der(C)$ in a standard way and the horizontal $(1,0)$-differential $d_h=[j(d_C),d_v]$ is given by the graded commutator with the contraction operator $j(d_C)$ by $d_C \in Der(C)$.
%
Seen as a dga with the total grading, $W(A)$  turns out to be semi-free, and therefore itself a Chevalley-Eilenberg algebra of an underlying {\bf shifted tangent} $\T A$ of $A$,
$CE(\T A) = W(A)$.
\end{definition}

\begin{remark}[Local structure]\label{rmk:localW}
There is a short exact sequence of $C$-modules 
which we now recall:
$$0 \to \Gamma(A)\otimes_{C^\infty(M)}C
\to Der(C)
\xto\alpha
\mathfrak X_M\otimes_{C^\infty(M)}C
\to 0,$$ Since $A$ is degreewise locally free, we can identifty $\Gamma(A)\otimes_{C^\infty(M)}C\cong Hom_{C^\infty(M)}(\Gamma(A^*),C)$ and 
$\mathfrak X_M\otimes_{C^\infty(M)}C\cong Hom_{C^\infty(M)}(\Omega^1(M),C)$.
Then $\alpha(X)(df) = X(f)$, and if $\alpha(X)=0$, then the derivation $X:C\to C$ is $C^\infty(M)$-linear, and it is completely determined by its value on the indecomposables. 
Dualizing over $C$, we obtain the short exact sequence
$$0 \to \Omega^1_M \otimes_{C^\infty(M)} C
\to \Omega^1_C
\to \Gamma(A^*) \otimes_{C^\infty(M)} C
\to 0$$
A choice of splitting for this sequence yields a local description of the shifted tangent. If $\{\xi_n^{\ell}\}$ are generators for $\Gamma A^*$ of degree $n$ over an open set $U \subset M$, then $\T A$ is
locally spanned by $\Omega^1_U$ and the $\xi_1^{\bullet}$ in total degree $1$, and by $\xi_n^{\bullet}$ together with
$d_v\xi_{n-1}^{\bullet}$ in total degree $n$. See, for instance, \cite{JMP}.
\end{remark}

\begin{example}
The Weil algebra of a manifold, viewed as a 0-algebroid, is its de Rham complex. If $A=\h$ is a Lie algebra, then $W^{p,q}(\h)=\Lambda^{p-q} \h^* \otimes S^q \h^*$ reproduces the classical construction. More generally, the Weil algebra of a Lie 1-algebroid has been studied by Abad-Crainic \cite{ac11} and Li Bland-Meinrenken \cite{LBM}, among others.
\end{example}

\begin{remark}\label{rmk:superWA}
From the point of view of supergeometry, recall that $A$ corresponds to an $\N$-graded supermanifold $\mathcal{M}_A$ with  $C^\infty(\mathcal{M}_A)=CE(A)$ and $d$ seen as a vector field $v$. Then, $W(A)= \Omega(\mathcal{M}_A)$ corresponds to differential forms, the horizontal differential $d_h=L_v$ is the Lie derivative along $v$, and the vertical differential $d_v=\d$ is the de Rham differential. Also notice that $\M_{\T A}=T[1]\M_A$ yields the shifted tangent bundle. More details and other explicit descriptions in terms of classical geometry can be found in \cite{FRS11,JMP}. (See also \cite{cuesch}.)
\end{remark}


Given $G$ a simplicial manifold, write $A=A_G$ for its higher Lie algebroid introduced in the previous subsection. We will show that the cosimplicial de Rham algebra $\Omega(G)$ and the Weil algebra $W(A)$ are related by a differentiation process completely analogous to the one for functions. For each $n>0$, let $I^\Omega_{n,j}=\ker s^*_{j-1}\subset \Omega(G_n)$ be the differential ideal given by the forms $\omega$ such that $s_{j-1}^*\omega=0$.

\begin{definition}
A normalized form $\omega\in \Omega_N(G_n)$ has {\bf higher vanishing} at $S_j$ if $\omega\in (I^\Omega_{n,j})^2$. The forms with higher vanishing at some $S_j$ span the degree $n$ higher vanishing ideal $J^\Omega_n$. The {\bf differentiating ideal} $\hat J^\Omega\subset \Omega_N(G)$ is the differential ideal spanned by $J^\Omega=\bigoplus_n J^\Omega_n$.
$$ J^\Omega_n = \bigoplus_j \Omega_N(G_n)\cap (I^\Omega_{n,j})^2 \subset \Omega_N(G_n) \qquad J^\Omega=\bigoplus_n J^\Omega_n \qquad \hat J^\Omega=J^\Omega+\delta(J^\Omega)$$
\end{definition}

Just as for $J$, it is easy to see that $J^\Omega$  is closed under the cup product so that, indeed, $J^\Omega$ only needs to be closed under $\delta$.  If $s: N\subset M$ is a closed embedded submanifold and $K=\ker s^*:\Omega(M)\to\Omega(N)$ is the ideal of forms with vanishing restriction, then $K$ is spanned, as a differential graded ideal, by its degree 0 part.
The proof is similar to that of Proposition \ref{prop:generators}. 
In particular, $\Omega_N(G_n)/J^\Omega_n$ localizes around $G_0\subset G_n$. Thus, we can fix a frame, work on the linear model, and consider normal coordinates $x_{\alpha,\ell}$, where $\alpha\subset\{1,\dots,n\}$ and $\ell\in\Gamma A_{|\alpha|}^*$. Then 
$$x_{\alpha,\ell}\in I^\Omega_{n,j} \iff dx_{\alpha,\ell}\in I^\Omega_{n,j} \iff j\in\alpha$$
It follows that $I^\Omega_{n,j}$ is spanned as a $C^\infty(G_0)$-algebra by the 0-forms $x_{\alpha,\ell}$ and the 1-forms $dx_{\alpha,\ell}$,  
$\Omega_{N}(G_n)$ is spanned by the normal monomials 
$$x_{P,L}dx_{P',L'}=x_{\alpha_1,\ell_1}\dots x_{\alpha_r,\ell_r}dx_{\alpha'_1,\ell'_1}\dots dx_{\alpha'_{r'},\ell'_{r'}}$$
with the juxtaposition $P\ast P'$ being a covering, and $J^\Omega_n$ is spanned by the normal monomials $x_{P,L}dx_{P',L'}$ with $P\ast P'$ a covering with overlap.


Let us denote $\ve^\Omega: \Omega_N(G) \to \Omega_N(G)/\J^\Omega$ the quotient map.
From the previous subsection, in form degree zero we have that $\ve^\Omega$ restricts to $\ve:C_N(G_p) \to CE^p(A_G)=W^{p,0}(A_G)$. Taking into account de Rham differential on $\Omega_N(G)$ and the vertical differential $d_v$ on $W(A_G)$, we can induce an identification $\Omega_N(G)/\J^\Omega \simeq W(A_G)$ on higher form degrees $q>0$ by $\ve^\Omega(d\omega) \simeq d_v \ve^\Omega(\omega)$ and by respecting the products, starting from $\ve^\Omega(f)=\ve(f)$ for $0$-forms $\omega=f$.
To illustrate this identification, consider local generators $\{\xi^\ell_n\}$ of degree $n$ for $CE(A_G)$, as introduced in Remark \ref{rmk:localW}. Writting them as $\xi^\ell_n = \ve(x_{\iota_n,\ell})$ for suitable normal coordinates, then 
\[\ve^\Omega(x_{P,L} dx_{P',L'}) \simeq \xi_{n_1}^{\ell_1} \cdots \xi_{n_r}^{\ell_r} d_v \xi_{n'_1}^{\ell'_1} \cdots d_v \xi_{n'_{r'}}^{\ell'_{r'}} \in W^{p,r'}(A_G) \]
for $P,P'$ canonical partitions with $P\ast P'$ covering $\{1,\dots,p=r+r'\}$, with $x_{P,L} dx_{P',L'}$ as described above, and where the product on the right is the one on $W(A_G)$. It is interesting to note that the commutation relations $x_{\iota_n,\ell}\cup x_{\iota_m,\ell'} \equiv_{\J} (-1)^{nm} x_{\iota_m,\ell'}\cup x_{\iota_n,\ell}$ found in the previous subsection for $0$-forms now get complemented by 
\[ dx_{\iota_n,\ell}\cup dx_{\iota_m,\ell'} \equiv_{\J^\Omega} (-1)^{(n+1)(m+1)} dx_{\iota_m,\ell'}\cup dx_{\iota_n,\ell}, \  \]
where now $\cup$ is induced by $\wedge$ in $\Omega(G)$. The proof of these facts follow straightforwardly from the $0$-form case.

Finally, since $\ve$ is surjective and $W^{p,0}(A)$ generates all of $W(A)$ under product and vertical differential, we thus obtain the following extension of the differentiation theorem for cochains.

\begin{corollary}\label{cor:thm2forms}
Let $G$ be a simplicial manifold and $\hat J^\Omega \subset \Omega(G)$ its differentiating ideal at the level of forms. Then, there is a natural isomorphism 
of differential bigraded algebras, 
$$\Omega^q_N(G_p)/\hat J^\Omega \simeq W^{p,q}(A_G)$$
where $W(A_G)$ is the Weil algebra of the higher Lie algebroid $A_G$ of $G$.
\end{corollary}






\begin{remark}
When $G$ the nerve of a groupoid $H\rightrightarrows M$, comparison results with existing formulae (see \cite{ac11,LBM}), analogous to that of Prop. \ref{prop:gdcomparison}, can be also carried out at the level of forms. This will be detailed elsewhere.
\end{remark}

\section{The higher van Est isomorphism theorem}\label{sec:vE}


In this section, starting with a simplicial manifold $G$, we consider our van Est map relating simplicial and infinitesimal cochains, 
$$\ve:C_N(G)\to CE(A_G)= C_N(G)/\J,$$ 
study the induced map in cohomology, and prove a van Est isomorphism theorem under suitable horn-filling and topological conditions (Theorem \ref{thm:3} below) generalizing classical results \cite{vanEst,crainic} (and also particular cases \cite{cabrera-drummond,angulo-cueca}).


\subsection{The decalage and the interpolating double complex}

Here we introduce a double complex $D$, which incorporates mixed horizontal-finite and vertical-infinitesimal cochains,
and is horizontally and vertically augmented by our complexes of interest. It is built using the double decalage $Dec(G)$, an iterated path space that we recall below. We also show that it has acyclic rows, yielding a factorization of the van Est map in cohomology.
$$
\xymatrix{ CE(A_G) \ar[r] & D \\ &\ar[u] C(G) } 
\qquad\qquad
\xymatrix{
H(A_G) \ar[r]^\cong & H(Tot D)  \\ & H(G) \ar[u]
\ar@{-->}[ul]^{\ve}}
$$



Given $(G,d_i,s_j)$ a simplicial manifold, its {\bf path space} $(P_0G,d'_i,s'_j)$ is the simplicial manifold given by $P_0G_n=G_{n+1}$, $d'_i=d_{i+1}$ and $s'_j=s_{j+1}$. 
We may think of $P_0G$ as the space of 1-simplices or paths in $G$ and homotopies fixing the source. 
Regarding $G_0$ as a constant simplicial manifold, we have a simplicial inclusion $s:G_0\to P_0G$, as well as a simplicial projection $e:P_0G\to G_0$ which keeps the first vertex. Then $es=\id$, and there is a simplicial homotopy $c:se\then \id$, the {\bf canonical contraction}, induced by the extra degeneracy: 
$$c_k:P_0G_n\to P_0G_{n+1} \qquad c_k(x)=(s_0)^{k+1}(d_1)^k(x) \qquad x\in P_0G_n \qquad 0\leq k\leq n$$ 
Note then that $H^k(P_0G)=0$ for $0<k$ and $H^0(P_0G)\cong C^\infty(G_0)$. The map $d_0:P_0G\to G$ is simplicial and resembles the path fibration.
$$\xymatrix{ & P_0G \ar[dl]_{d_0} \ar[dr]^{e} & \\ G & & G_0 \ar@/^/[lu]^{s}}$$

\begin{example}
If $G$ is the nerve of Lie groupoid $H\toto M$, then $d_0:P_0G\to G$ is a simplicial model for the universal principal $H$-bundle $EH\to BH$. If $H$ is a Lie group, then $P_0G=EH$ is contractible.
\end{example}

The {\bf decalage} $Dec(G)$ \cite{illusie}, see also \cite{cegarra}, is the bisimplicial manifold obtained by iterating the path space construction. Starting with a simplicial manifold $G$, we have
$$Dec(G)_{pq}=G^{\Delta^p\ast\Delta^q}\cong G_{p+q+1} \qquad 
\begin{cases}
d_i^h(g)=d_{i+q+1}(g) & d_i^v(g)=d_i(g) \\ 
s_j^h(g)=s_{j+q+1}(g) & s_j^v(g)=s_j(g)
\end{cases}
$$

A $(p,q)$-simplex $g\in Dec(G)_{pq}$ is a simplicial map $\Delta^p\ast\Delta^q\to G$, where $\Delta^p\ast\Delta^q$ is the simplicial join, which is canonically isomorphic to $\Delta^{p+q+1}$. We can visualize its vertices in two rows, the upper one with $p+1$ vertices and the lower one with $q+1$ vertices:
$$\xymatrix{ g_{p+q+1} & g_{p+q} & \cdots & g_{q+2} & g_{q+1} }$$
$$\xymatrix{g_{q} & g_{q-1} & \cdots & g_1 & g_0 }$$
Then $d^h_i(g)$ is the face missing the $i$-th vertex of the top row, while $d^v_i(g)$ is the one missing the $i$-th vertex of the bottom row. Similarly, $s^h_j(g)$ repeats the $j$-th vertex of the top row, while $s^v_j(g)$ repeats the $j$-th vertex of the bottom row. We can depict the whole situation as follows:

{\small
$$\xymatrix{ 
\ar@{}[r]|{}="a"  
\ar@{}[ddd];[ddddr]|{}="b"
\ar@{}[dddrrrr];[ddddrrrr]|{}="c" 
\ar@{--}"a";"b"    \ar@{--}"b";"c"
\vdots \ar@<2pt>[d] \ar@<-2pt>[d] \ar@<6pt>[d] \ar@<-6pt>[d] & \vdots \ar@<2pt>[d] \ar@<-2pt>[d] \ar@<6pt>[d] \ar@<-6pt>[d] & \vdots \ar@<2pt>[d] \ar@<-2pt>[d] \ar@<6pt>[d] \ar@<-6pt>[d] & \vdots \ar@<2pt>[d] \ar@<-2pt>[d] \ar@<6pt>[d] \ar@<-6pt>[d] & \reflectbox{$\ddots$} \\
G_2 \ar@<5pt>[d] \ar[d] \ar@<-5pt>[d] & \ar[l] Dec(G)_{02} \ar@<5pt>[d] \ar[d] \ar@<-5pt>[d] & Dec(G)_{12} 
\ar@<3pt>[l] \ar@<-3pt>[l] \ar@<5pt>[d] \ar[d] \ar@<-5pt>[d]
& Dec(G)_{22} \ar@<5pt>[d] \ar[d] \ar@<-5pt>[d] \ar@<5pt>[l] \ar[l] \ar@<-5pt>[l] & \dots \ar@<2pt>[l] \ar@<-2pt>[l] \ar@<6pt>[l] \ar@<-6pt>[l]\\
G_1 \ar@<3pt>[d] \ar@<-3pt>[d] & \ar[l] Dec(G)_{01} \ar@<3pt>[d] \ar@<-3pt>[d] & Dec(G)_{11} \ar@<3pt>[l] \ar@<-3pt>[l] \ar@<3pt>[d] \ar@<-3pt>[d] & Dec(G)_{21} \ar@<3pt>[d] \ar@<-3pt>[d] \ar@<5pt>[l] \ar[l] \ar@<-5pt>[l] & \dots \ar@<2pt>[l] \ar@<-2pt>[l] \ar@<6pt>[l] \ar@<-6pt>[l] \\
G_0 & \ar[l] Dec(G)_{00} \ar[d]& Dec(G)_{10} \ar[d] \ar@<3pt>[l] \ar@<-3pt>[l] & Dec(G)_{20} \ar[d] \ar@<5pt>[l] \ar[l] \ar@<-5pt>[l] &  \dots \ar@<2pt>[l] \ar@<-2pt>[l] \ar@<6pt>[l] \ar@<-6pt>[l] \\
& G_0 & G_1 \ar@<3pt>[l] \ar@<-3pt>[l] & G_2 \ar@<5pt>[l] \ar[l] \ar@<-5pt>[l] & \dots \ar@<2pt>[l] \ar@<-2pt>[l] \ar@<6pt>[l] \ar@<-6pt>[l]
}$$
}


\begin{remark}\label{rmk:contractions}
The decalage has canonical horizontal and vertical augmentations and contractions
$$d_{q+1}:Dec(G)_{0q}\to G_q \qquad d_0:Dec(G)_{p0}\to G_p$$
$$c^h_k:Dec(G)_{p,q}\to Dec(G)_{p+1,q} \qquad
c^v_k:Dec(G)_{p,q}\to Dec(G)_{p,q+1}$$ 
which are those of the underlying path spaces.  The contraction $c^h_k:Dec(G)_{p,q}\to Dec(G)_{p+1,q}$, $c^h_k(x)=s_q^{k+1}d_{q+1}^k(x)$, $x\in G_{p+q+1}$, $0\leq k\leq p$, is not a simplicial map between the columns 
for it fails to commute with $d_q^v$, but it does commute with the other faces and the degeneracies. Concretely, if $i<q$, we have 
$d^v_i c^h_k=c^h_k d^v_i:Dec(G)_{p,q}\to Dec(G)_{p+1,q-1}$, and $s^v_j c^h_k=c^h_k s^v_j:Dec(G)_{p,q}\to Dec(G)_{p+1,q+1}$
The situation with the vertical contractions is analogous.
\end{remark}

The decalage $Dec(G)$ yields a double complex
$C_{pq}=(C^\infty(Dec(G)_{p,q}),\delta^h,\delta^v)$
with acyclic rows and columns.
Each column is the algebra of cochains on an iterated path space, $C_{p,\bullet}=C(P_1^{p+1}G)$, so it has its own subalgebra of normalized functions $N_{p,q}^{v}$ and its own ideal of higher vanishing functions ${\Jgen}_{p,q}^v$. The vertical differentiating ideal is then $\J^v = {\Jgen}^v + \delta^v {\Jgen}^v$.
The horizontal differential $\delta^h$ restricts to the double subcomplex $N^v$ and descends to the quotient $N^v/\J^v$. 

\begin{definition}\label{def:D}
Given $G$ a simplicial manifold, its {\bf \nameD} $(D,d^v,\delta^h)$ is the vertical differentiation of the decalage cochains, namely $D_{pq} = N^v_{pq}/\J^v_{pq}$,
$d^v$ is the vertical Chevalley-Eilenberg differential, and $\delta^h$ is the residual horizontal differential.
\end{definition}


\begin{proposition}\label{prop:rows}
The horizontally augmented double complex $CE(A_G)\to D$ has exact rows. The augmentation yields an isomorphism in cohomology $H(A_G)\xto\cong H(D)$
\end{proposition}

\begin{proof}
We define the map $\theta: C_{p,q}\to  C_{p-1,q}$, $\theta(f)=\sum_{k=0}^{p-1} (-1)^{q+k}(c^h_k)^*(f)$. 
Since the horizontal contractions commute with the vertical degeneracies, each pullback operator $(c^h_k)^*$ preserves $N^v$ and $\Jgen^v$, and therefore, so does $\theta$. 
We claim that $\theta(\Jgen^v+\delta^v(\Jgen^v))\subset \Jgen^v+\delta^v(\Jgen^v)$, and therefore, $\theta$ descends to a horizontal contraction of the double complex $D$. 

Let $f\in \Jgen^v_{p,q-1}$ with higher vanishing over $S^v_{j_0}=s^v_{j_0-1}(Dec(G)_{p,q-2})$, $1\leq j_0\leq q-1$,
and let us show that $\theta\delta^v(f)\subset J^v+\delta^v(J^v)$.
We have 
$$\theta\delta^v(f)=
\theta\delta^v(f)-\delta^v\theta(f)+\delta^v\theta(f)=
\sum_{k=0}^{p-1}\sum_{i=0}^q  
(-1)^{q+k+i}(fc_k^hd^v_i -fd^v_ic_k^h) +\delta^v\theta(f)$$
The terms corresponding to $i<q$ cancel out for, in this case, the vertical face commutes with the horizontal contraction, see Remark \ref{rmk:contractions}. Then we have
$$\theta\delta^v(f)=
\sum_{k=0}^{p-1}(-1)^{k}(fc_k^hd^v_q -fd^v_qc_k^h) +\delta^v\theta(f)$$
We know that $\theta(f)\in \Jgen^v$, and therefore $\delta^v\theta(f)\in \delta^v(\Jgen^v)$. Moreover, each of the terms $(c_k^hd_q^v)^*(f)$ and $(d_q^vc_k^h)^*(f)$ are in $\Jgen^v$. This follows from the fact that $c_k^hd_q^v(S^v_j)\subset S^v_j$ and 
$d_q^vc_k^h(S^v_j)\subset S^v_j$ for $1\leq j\leq q-1$. And this is because, since $j<q$, $c^h_kd^v_qs^v_{j-1}= c^h_k s^v_{j-1} d^v_{q-1} = s^v_{j-1}c^h_k  d^v_{q-1}$, and similarly for the other case.
Hence we have proved that the chain homotopy contraction $\theta$ descends to the quotients, and that after vertical differentiation, the rows are still exact. Thus, the cohomology of the algebroid is the same as that of $D$.
\end{proof}


\begin{remark}\label{rmk:horizVE}
As claimed at the beginning of the subsection, the interpolation is compatible with the van Est map at the level of cohomology. At a first glance, the van Est map $\ve:C_N(G)\to CE(A_G)$ does not commute with the augmentations of $D$: given $f\in C^k_N(G)$, it appears in $D_{k,0}$ via the vertical augmentation, while $\ve(f)$ appears in $D_{0,k}$ via the horizontal augmentation
$$
\begin{matrix}\xymatrix{
CE(A_G) \ar[r] & D \\ C_N(G) \ar[r]^\subset \ar[u]^\ve & C(G) \ar[u]
}\end{matrix}\qquad\overset{H}\mapsto\qquad
\begin{matrix}\xymatrix{
H(A_G) \ar[r]^\cong & H(D)  \\H(G) \ar@{=}[r] \ar[u]^\ve & H(G) \ar[u]}
\end{matrix}$$ 
Nevertheless, these two elements are cohomologous in $H(D)$, as follows. Write $C_{pq}=C^\infty(Dec(G)_{pq})$ and recall the vertical and horizontal augmentations, $d_0^*:C^p(G)\to C_{p,0}$ and $d_{last}^*:C^p(G) \to C_{0,q}$, respectively.
Given $f\in C^\infty(G_k)$, $k\geq 0$, write $(f,p,q)\in C^{pq}$ for the function $f$ viewed as a cochain in $C$ of bidegree $(p,q)$, $p+q=k-1$. If  $\delta f = 0$, in the total complex we have 
$$\delta^{tot}\left(\sum_{p+q=k-1} (f,p,q)\right)= d_0^*f-d_{last}^*f,$$
where $\delta^{tot}=\delta^v + (-1)^{q}\delta^h$ is the total differential. We also note that, in the case of $1$-groupoids, \cite{LBM} uses a similar interpolating double complex and the van Est map is defined at the level of cochains via a homotopy inverse for the horizontal augmentation. In our approach, the idea is that we already have the van Est map naturally defined through a geometric construction and we shall not need such a homological characterization. Also notice that, although our $\ve$ is defined over the normalized cochains $C_N(G)$, it can be extended to all cochain just by composing with the canonical retraction $r:C(G)\to C_N(G)$, 
$r(f)=\sum_{\{i_1<\dots<i_k\}\subset\{1,\dots,n\}} (-1)^{k}f s_{i_1}\dots s_{i_k} d_{i_k}\dots d_{i_1}$. 
\end{remark}


\subsection{Hypercovers at the infinitesimal level}

\label{subsec:hypercovers}


We now analyze the cohomology of the columns of $D$.  When $G$ is a higher groupoid, the augmentations $Dec(G)_{p\bullet}\to G_p$ turn out to be hypercovers, which are simplicial fibrations generalizing the groupoid Morita equivalences (see \cite{getzler,dhos}). In this subsection, we show that general hypercovers yield quasi-isomorphisms at the level of Chevalley-Eilenberg algebras, by working with simple hypercovers. In the following subsection, we return to the decalage and apply this general result to prove the van Est theorem.


We recall some notions and conventions from \cite{dhs}. A simplicial set $K$ can be regarded as the colimit of its non-degenerate simplices, $K=\colim_{\Delta^n\mono K}\Delta^n$. Then, if $G$ is a simplicial manifold, the space of maps 
$K\to G$ is a limit space in a canonical way, namely, a subspace of the ambient product
$$G^K=\lim_{\Delta^n\mono K} G_n\subset\prod_{\Delta^n\mono K}G_n$$
When $K$ has finitely many non-degenerate simplices, then we say that $G^K$ is smooth if it is an embedded submanifold of the ambient product. Similarly, given $i:K\to L$ a map of finite simplicial sets and $f:G'\to G$ a map of simplicial manifolds, the space $\hom(i,f)= G'^K\times_{G^K} G^L$ of commutative squares
$$\xymatrix{K \ar[r] \ar[d] & G' \ar[d] \\ L\ar[r] & G}$$
is smooth if it is and embedded submanifold of the product $\prod_{\Delta^n\mono K}G'_n\times \prod_{\Delta^n\mono L}G_n$.
Given $f:G'\to G$, the {\bf relative matching spaces} are defined as $M_n(G'/G)=(G')^{\partial \Delta^n}\times_{G^{\partial\Delta^n}}G_n$.

\begin{definition}
A morphism $f:G'\to G$ of simplicial manifolds is a {\bf hypercover} if the relative matching spaces $M_n(G'/G)$ are smooth and the natural restriction maps $\partial_n:G'_n\to M_n(G'/G)$ are surjective submersions for all $n\geq 0$.
\end{definition}

\begin{example}
A hypercover between Lie 0-groupoids is the same as a diffeomorphism between the corresponding manifolds, and a 
hypercover between Lie 1-groupoids is the same as a Morita fibration, see e.g. \cite{survey}. 
\end{example}

\begin{remark}
In the case $n=0$, since $\partial\Delta^0=\emptyset$, 
the restriction map $\partial_0:G'_0\to M_0(G'/G)=G_0$ identifies with $f_0$. This is not required, for instance, in  \cite{wolfson}. We follow the notations and conventions in \cite{dhos}. For us, hypercovers are simplicial fibrations, and they are levelwise surjective submersions. 
\end{remark}


The next lemma explains how hypercovers naturally appear as path fibrations in the decalage construction.

\begin{lemma}
If $G$ is a higher Lie groupoid, then so is $P_0G$, and the first vertex projection $e:P_0G\to G_0$ is a hypercover.
\end{lemma}

\begin{proof}
To see that $P_0G$ is a higher Lie groupoid, the idea is to fill a $k$-horn in the path space is the same as filling a simplex out of a map $K\to G$, where $K$ is obtained from $\Lambda^n_0$ by removing the $k+1$-th face. More precisely, we have natural identifications $P_0G=G_{n+1}$ and $(P_0G)^{\Lambda^{n+1}_k}=G^K$, and the claim follows from \cite[Corollary 6.6]{dhs}.  To see that $e$ is a hypercover, we need to show that the $(n,0)$-horns of $G$ admit a filling, and this holds because $G$ is a higher Lie groupoid.
\end{proof}


The class of hypercovers contains the isomorphisms and is closed under base-change and composition \cite{dhs}. Moreover, to understand general hypercovers, it is enough to pay attention to a very specific type. A hypercover $f:G'\to G$ is {\bf $m$-simple} if $\partial_n$ is a diffeomorphism for every $n\neq m$. In such a hypercover, the maps $f_k:G'_k\to G_k$ are isomorphisms for $k<m$, and we can think of $G'$ as a sort of $m$-shifted pullback of $G$ through $f_m$. The following is proven \cite{dhs} for higher Lie groupoids with bounded tangent complex, but the same proof applies to the general case.


\begin{proposition}\cite{dhs}
If $f:G'\to G$ is a hypercover between higher Lie groupoids, 
the {\bf relative $m$-coskeleton} $\cosk_m(G'/G)=(G')^{\sk_k\Delta^n}\times_{G^{\sk_k\Delta^n}}G^{\Delta^n}$, whose $n$-simplices are commutative squares
$$\xymatrix{
\sk_m\Delta^n \ar[r] \ar[d] & G' \ar[d] \\
\Delta^n \ar[r] & G,
}$$ 
is a higher Lie groupoid, and the natural restriction $f^m:\cosk_m(G'/G)\to \cosk_{m-1}(G'/G)$ is an $m$-simple hypercover for every $m$.
\end{proposition}


The {\bf coskeletal tower} of a hypercover $f:G'\to G$ is the following factorization:
$$G'\to \dots \to \cosk_{m+1}(G'/G)\xto{f^{m+1}} \cosk_{m}(G'/G)\xto {f^m}\cdots \xto{f^1}\cosk_{0}(G'/G)\xto{f^0} G$$
Even though the tower may be infinite, each finite truncation stabilizes, since $f^m_n:\cosk_m(G'/G)_n\to \cosk_{m-1}(G'/G)_n$ is an isomorphism whenever $n<m$. The cohomological descent for hypercovers proves that each $f^m$, and therefore $f$, yield isomorphism in cohomology, see \cite{dhos}. We will show that for $m>0$ the same holds at the infinitesimal level. 

\begin{proposition}\label{prop:hypercover}
Let $f:G'\to G$ be an $m$-simple hypercover, $m>0$, let $V=\ker Df_m|_{G'_0}$ and let $D^{m}(V^*)$ denote the complex $V^*\xto\id V^*$ concentrated in degrees $m$ and $m+1$. Then:
\begin{enumerate}[a)]
\item The cotangent complexes fit into a short exact sequence
$0\to A_G^*\to A_{G'}^*\to D^{m}(V^*)\to 0$,
\item There is a dga isomorphism
$CE(A_{G'})\cong CE(A_G)\otimes S(D^{m}(V^*))$, and
\item The induced morphism $CE(A_G)\to CE(A_{G'})$ is a quasi-isomorphism.
\end{enumerate}
\end{proposition}

\begin{proof}
Let us prove the dual statement to a) regarding tangent complexes. We can focus on germs around $G_0'\simeq G_0$, and split a horn-filling of $G'$ into two steps: first fill its projection in $G$, then lift the filling. With this a) becomes clear and the details are left to the reader.

We next prove part b), writing $A=A_G$ and $A'=A_{G'}$, and starting with the case in which $f$ admits a simplicial section $g: G\to G'$. This gives a splitting of the above sequence, and a retraction for the dga morphism $f^*:CE(A)\to CE(A')$. Then, the differential in $CE(A')$ splits, and we have a dga isomorphism
$$CE(A')\cong CE(A)\otimes S(D^{m+1}(V^*))$$
For the general case, we will show that, at least around $G_0$, there is such a simplicial section $g:(G,G_0)\to (G',G_0)$, and the result follows by locality. We define $g_n$ inductively. If $n<m$ then $G'_n\cong G_n$ is a diffeomorphism and we set $g_n=f_n^{-1}$. Then we construct $g_m:(G_m,G_0,e)\to (G'_m,G_0,e)$ to be a microbundle retraction to $f_m$ whose image contains (the germ of) the degeneracy locus. This can be achieved exactly as in the proof of Theorem \ref{thm:1}, by first fixing a frame in $G$, and then extending the partial sections $s'_{j-1}s_{j-1}^{-1}:(S_{m,j},G_0,e)\to (G'_m,G_0,e')$ via Lemma \ref{lemma:nice-family}. 


For $n>m$, consider the map giving the restriction to the $m$-dimensional skeleta, denoted $G'_n \to \prod_{\alpha:[m]\mono[n]}G'_m$. Since $f$ is $m$-simple, an $n$-simplex on $G'$ is determined by the corresponding $n$-simplex on $G$ and $f$-lifts of its $m$-skeleta: the following is a cartesian square,
$$\xymatrix{
G'_n \ar[r] \ar[d] & \prod_{\alpha:[m]\mono[n]}G'_m \ar[d] \\
G_n \ar[r] & \prod_{\alpha:[m]\mono[n]}G_m
}$$
We then define $g_n$ as the pullback section of $\prod_\alpha g_m$ on the right.
%
%
%
It is straightforward to check that the $g_n$ thus defined commute with the faces and the degeneracies, hence inducing the desired dga isomorphism.

Finally, c) follows from b), for $CE(A)\to CE(A')$ is a semi-free acyclic extension, and therefore a quasi-isomorphism, see e.g. \cite[Lemma V.3.7]{gelfand-manin}.
\end{proof}


\begin{corollary}[Easy van Est]\label{cor:easyVE}
Let $G$ be a higher Lie groupoid, suppose that there exists a hypercover $f:G\to M$ to a constant simplicial manifold $M$, and that the fibers of $f_0:G_0 \to M$ are homologically $n$-connected. Then,  
$$H^k(G)\cong H^k(A_G) = \begin{cases}
C^\infty(M),& k=0 \\
0, & 0<k\leq n.
\end{cases} $$
\end{corollary}

\begin{proof} 
The coskeletal tower of $f$ induces the following diagram of dgas:
$$\xymatrix@C=18pt{
\dots & C_N(\cosk_{m}(G/M)) \ar[d]^{\ve} \ar[l] & \cdots \ar[l] & C_N(\cosk_{0}(G/M)) \ar[d]^{\ve} \ar[l] & C_N(M)=C^\infty(M) \ar[d] \ar[l] \\
\dots  & CE(A_{\cosk_{m}(G/M)}) \ar[l] & \cdots \ar[l] & CE(A_{\cosk_{0}(G/M)}) \ar[l] & CE(0_M)=C^\infty(M) \ar[l] }
$$
The upper arrows are quasi-isomorphisms by the cohomological descent for hypercover, see e.g. \cite{dhos}. The bottom arrows are quasi-isomorphisms, except perhaps the first one, by Proposition \ref{prop:hypercover}.
Moreover, the projection $G\to \cosk_{m}(G/M)$ yields an isomorphism in $k$-cochains, for $k<m$, and therefore, it yields isomorphisms in cohomology 
$$H^k(\cosk_{m}(G/M))\cong H^k(G) \qquad 
H^k(A_{\cosk_{m}(G/M)})\cong H^k(A_G) \qquad k<m-1$$
Writing $G'=\cosk_{0}(G/M)$, we are thus left with showing that $H^k(A_{G'})\cong H^k(M)$ for $0\leq k \leq n$. We use the topological hypothesis. Note that $G'$ is the Lie 1-groupoid induced by the submersion $f_0:G_0\to M$, and its source fibers are those of $f_0$, so we can apply the classical van Est for Lie groupoids and Lie algebroids, \cite[Thm 4]{crainic}, and the proof is completed.
\end{proof}

\subsection{Connectedness, the van Est theorem and applications}

In this final subsection, we describe the topological assumption, state and prove the higher van Est theorem, and provide an application to the Lie theoretic description of shifted symplectic structures.


The following connectedness definition for simplicial manifolds is inspired by the classical van Est for Lie groupoids, see \cite[Thm 4]{crainic}.

\begin{definition} \label{def:connected}
    We say that a simplicial manifold $G$ is {\bf $n$-connected} when the fibers of $d_0:G_{p+1}\to G_p$ are homologically $(n-p)$-connected, for all $0\leq p\leq n$.
\end{definition}

\begin{example}\label{ex:1gdconnected}
If $G$ is the nerve of a Lie groupoid $H\toto M$, the map $d_0:G_{p+1}\to G_p$ is a base-change of the target $t:H \to M$, so they have the same fibers. In a Lie groupoid, source and target fibers are diffeomorphic, so $G$ is $n$-connected if and only if $H$ is homologically source-$n$-connected. Thus, our condition recovers that of classical van Est \cite[Thm 4]{crainic}.
\end{example}


Connectedness of a higher Lie groupoid can be deduced from the fibers of a horn map, which are typically lower dimensional than those of $d_0$, thus providing a more practical characterization.

\begin{lemma} 
If $G$ is a higher Lie groupoid such that the fibers of the last horn map $d_{p+1,p+1}:G_{p+1}\to G_{p+1,p+1}$ are homologically $(n-p)$ connected for all $0\leq p\leq n$, then $G$ is $n$-connected.
\end{lemma}

\begin{proof} 
We can go from $\Delta^p$ to its 0-face by a sequence of elementary collapses along a last horn, 
$$\Delta^p=F_a\supset \Lambda^p_p=F_{a-1}\supset\dots\supset F_r=F_{r-1}\cup_{\partial\Lambda^{n_r}_{n_r}}\Delta^{n_r}\supset F_{r-1}
\supset\dots\supset F_0=d_0(\Delta^p)\cong\Delta^{p-1}$$ 
For instance, take the set $\sigma_r\subset \Delta^p$ of all the non-degenerate simplices containing $0$ and $p$, order the $\sigma_r$ so that $\dim(\sigma_{r-1})\leq\dim(\sigma_r)$, and define $F_r=d_0(\Delta^p)\cup\sigma_1\cup\dots\cup\sigma_r$.
This yields a tower 
$$G_p \to G_{p,p}\to \dots\to 
G^{F_r}\to \dots\to d_0(G_p)\cong G_{p-1}$$
where the $r$-th step $G^{F_r}\to G^{F_{r-1}}$ is a base-change of the surjective submersion $G_{n_r}\to G^{\Lambda^{n_r}_{n_r}}$. 
Since submersions with homologically $k$-connected fibers are stable under base-change and compositions, the result follows.
\end{proof}


We are finally in a position to prove the van Est theorem for higher Lie groupoids. It is a far-reaching generalization of classical van Est for Lie groupoids \cite{crainic}, and even in the Lie groupoid case, our approach offers an alternative geometric proof building on our alternative characterization of infinitesimal cochains $CE(A_G)=C_N(G)/\J$.

\begin{theorem}[van Est isomorphism]\label{thm:3}
Let $G$ be an $n$-connected higher Lie groupoid, $n\geq 0$.
Then, the van Est map $\ve$ induces isomorphism in cohomology $H^k(G)\cong H^k(A_G)$ for every $k\leq n$.
\end{theorem}

\begin{proof}
We recall from the beginning of the section the following diagram:
$$
\xymatrix{ CE(A_G) \ar[r] & D \\ &\ar[u] C(G) } 
\qquad\qquad
\xymatrix{
H(A_G) \ar[r]^\cong & H(Tot D)  \\ & H(G) \ar[u]
\ar@{-->}[ul]^{\ve}}
$$
By Proposition \ref{prop:rows}, the rows of $D$ are exact, so $H(A_G)\xto\cong H(D)$. Now, by the Easy van Est \ref{cor:easyVE}, the double complex $D$ vertically augmented by $C(G)$ has trivial vertical cohomology in the triangular sector of bidegrees $p,q$ with $0<p+q\leq n$. The result now follows by a standard argument in double complexes, either by using spectral sequences, or by directly constructing a zig-zag isomorphism (see Remark \ref{rmk:horizVE}).
\end{proof}

\begin{remark} 
The hypothesis for our theorem admits simple variants and generalizations. For instance, for our proof to work, it is enough that $G$ is a simplicial manifold such that the $0$-horn spaces $G_{n,0}$ are smooth and the restriction maps $G_n\to G_{n,0}$ are surjective submersions. Besides, if we replace $G$ by its {\bf reverse} $G^{rev}$, given by $G^{rev}_n=G_n$, $d^{rev}_i=d_{n-i}:G_n\to G_{n-1}$ and $s^{rev}_j=s_{n-j}:G_n\to G_{n+1}$, we clearly have $H(G)\cong H(G^{rev})$, and we can achieve the theorem from the following alternative topological assumption: the last face map $d_{p}:G_{p}\to G_{p-1}$ has $(n-p)$-connected fibers for every $p$. 
\end{remark}

\begin{remark}[Integral formulas for 1-cocycles]
Let $G$ be a simplicial manifold and $f\in C^\infty_N(G_1)$ with $\delta f=0$. Denote $\Delta_1=[0,1]$ the geometric 1-simplex and $P(\Delta_1)$ the nerve of its pair groupoid whose differentiation yields differential forms $CE(A_{P(\Delta_1)})=\Omega(\Delta_1)$. For any smooth simplicial map $F_\bullet: P(\Delta_1)_\bullet \to G_\bullet$, we have the following integral formula
\[ f(F_1(1,0)) = \int_{\Delta_1} Lie(F)^*\ve(f). \]
This also allows to integrate infinitesimal 1-cocycles, generalizing a formula for Lie groupoids in \cite{wx}. Higher dimensional simplicial integral formulas, generalizing those of \cite[Sec. 2.5]{cms}, will be explored elsewhere. 
\end{remark}


There is an analogous theorem at the level of forms for the van Est map $\Omega(G)\to W(A_G)$ of Corollary \ref{cor:thm2forms}. 
Just as for defining $D$ above, we can consider the triple complex $C^{pqr} = \Omega^r(Dec(G)_{pq})$, where the extra differential is de Rham, and define $D^{pqr}$ to be its differentiation on the $q$-direction using the corresponding ideal $\J^\Omega$. 
For each $r$, we thus obtain an interpolating double complex analogous to $D$ between simplicial $r$-forms on $G$ and $(\bullet,r)$-elements of the Weil algebra of $A_G$,
$$ \xymatrix{ W^{\bullet r}(A_G) \ar[r] &  D^{\bullet \bullet r} \\ &\ar[u] \Omega^r(G) } $$
As in the case of cochains, the rows of $D^{\bullet \bullet r}$ are also exact, because of the horizontal homotopy, and the induced map in cohomology $H(\Omega^r(G))\to H(W^{\bullet,r}(A_G))$ coincides with the one induced by the van Est map. Under the same assumption of Corollary \ref{cor:easyVE}, one analogously obtains that
$$ H^k( D^{p\bullet r}) =  \begin{cases}
\Omega^r(G_p),& k=0 \\
0, & 0<k\leq n.
\end{cases} $$
We thus conclude the following:

\begin{corollary}\label{cor:vEisoforms}
If $G$ is an $n$-connected higher Lie groupoid, then 
for each fixed $r\geq 0$, the van Est map $\ve:\Omega^r(G_p) \to W^{p,r}(A_G)$ for $r$-forms induces an isomorphism in $(\bullet,r)$-cohomology $H^k(\Omega^r(G_\bullet))\simeq H^k(W^{\bullet,r}(A_G))$ for degrees $0\leq k \leq n$.
\end{corollary}

%


We finish this section applying our higher van Est theory to the Lie theoretic description of so-called \emph{shifted symplectic structures} \cite{getzler,PTVV13} (see also \cite{cueca-zhu,cumava}). 

Let $G$ be a higher Lie groupoid. An {\bf $m$-shifted $k$-form} on $G$ is an element $\omega \in Tot^{k+m}(\Omega_N^\bullet(G_\bullet))$ of degree $m+2$ in the total de Rham complex which is a sum of elements of bidegree $(p,r)$ with $r\geq k$. Denoting $\dtot$ the total differential, an {\bf $m$-shifted symplectic form} is an $m$-shifted $2$-form $\omega$ such that $\dtot\omega =0$ and satisfying a non-degeneracy condition \cite{getzler}, as follows. Writing $\omega_m$ for the bidegree $(m,2)$ component of $\omega$, the induced map between tangent and cotangent complexes $\lambda^{\omega}:(A_\bullet, \partial)\to (A^*[m]_\bullet,\partial^*)$ via  
\[ \langle \lambda^\omega(v),u\rangle = \underset{\sigma\in Sh(l,m-l)}{\sum} \text{sgn}(\sigma) \ \omega_m(D(s_{\sigma(m-1)}\cdots s_{\sigma(l)})v, D(s_{\sigma (l-1)} \cdots s_{\sigma(0)})u), \ v\in A_l, u\in A_{m-l}, \]
must be a quasi-isomorphism. 
A {\bf gauge transformation} (see \cite[Prop. 2.35]{cueca-zhu}) of an $m$-shifted symplectic form is $\omega \mapsto \omega+\dtot\eta$ where $\eta$ is an $(m-1)$-shifted $2$-form.
Similarly, at the infinitesimal level, consider $A$ a higher Lie algebroid with Weil double complex $W^{\bullet,\bullet}(A)$. An $m$-shifted (infinitesimal) $k$-form is an element in the total complex $\un{\omega} \in Tot^{k+m}(W^{\bullet,\bullet}(A))$ whose homogeneous components of bidegree $(p,r)$ have $r\geq k$. Denoting also by $\dtot$ the corresponding total differential, an  {\bf (infinitesimal) $m$-shifted symplectic form} (\cite{PySa}) is an $m$-shifted $2$-form $\un{\omega}$ on $A$ satisfying $\dtot\un{\omega}=0$ and a non-degeneracy condition stating that the induced map at the level of tangent and cotangent complexes, $\lambda^{\un{\omega}}:(A_\bullet, \partial)\to (A^*[m]_\bullet,\partial^*)$, is a quasi-isomorphism. Supergeometrically, we can regard $\un{\omega}$ as a $2$-form on the supermanifold $\M_A$ underlying $A$ (recall Remark \ref{rmk:superA}) and $\lambda^{\un{\omega}}$ is induced by restricting $\un{\omega}$ to the  tangent spaces at $M\subset \mathcal{M}_A$. 
A {\bf gauge transformation} of an $m$-shifted symplectic form is $\un{\omega}\mapsto \un{\omega} + \dtot\un{\eta}$ where $\un{\eta}$ is an $(m-1)$-shifted $2$-form on $A$.

\begin{corollary}\label{cor:shiftedsympl}
Let $G$ be a higher Lie groupoid with higher Lie algebroid $A_G$. The van Est map for forms, $\Omega^q(G_p)\to W^{p,q}(A_G)$, sends $m$-shifted symplectic forms on $G$ to $m$-shifted symplectic forms on $A_G$. Moreover, if $G$ is $m$-connected, this correspondence is $1:1$ modulo the corresponding gauge transformations,
\[ \{\text{$m$-shifted symplectic $\omega$ on $G$}\}/\text{gauge} \overset{\sim}{\to} \{\text{$m$-shifted symplectic $\un{\omega}$ on $A_G$}\}/\text{gauge}.\]
\end{corollary}

\begin{proof}
The proof is immediate from the previous results and can be summarized as follows. Given $\omega$ on $G$, then $\un{\omega}:=\ve^{\Omega}(\omega)$ is clearly a $\dtot$-closed $m$-shifted $2$-form on $A_G$, since $\ve^\Omega$ is a map of double complexes. Using the definition of $\ve^\Omega$, it follows directly that 
$\lambda^\omega = \lambda^{\un{\omega}}$
so that $\un{\omega}$ is non-degenerate. Indeed, since both maps $\lambda^\omega, \lambda^{\un{\omega}}$ are invariantly defined and pointwise with respect to $G_0$, it is enough to compare them in coordinates induced by any frame, where the equality easily follows. (See the formulas in Section \ref{subsec:forms} for $\ve^\Omega$ on generators.) 

Finally, assume $G$ is $m$-connected and let $\un{\omega}$ be an infinitesimal $m$-shifted symplectic form on $A_G$. Extending the horizontal isomorphisms of Corollary \ref{cor:vEisoforms} to the double complex in a standard way, it follows that there exists a closed $m$-shifted $2$-form $\omega$ on $G$ such that $\ve^\Omega(\omega)=\un{\omega}+\dtot\un{\eta}$ for $\un{\eta}$ of total degree $m+1$. Since $\ve^\Omega$ yields isomorphisms for the horizontal differentials, we can assume $\un{\eta}$ is an$(m-1)$-shifted $2$-form. On the other hand, since $\lambda^{\un{\omega}+\dtot\un{\eta}}$ is chain homotopic to the original $\lambda^{\un{\omega}}$ (\cite[Prop. 6.5]{cumava}) and using $\lambda^\omega=\lambda^{\un{\omega}}$ above, we have that $\omega$ is $m$-shifted symplectic on $G$. This shows that differentiation is surjective onto gauge equivalence classes. To show injectivity up to gauge in the domain, we observe that, by the isomorphism in total cohomology, another integration $\omega'$ must differ from $\omega$ by $\dtot\eta$ with an $\eta$ of total degree $m+1$. Again, since $\ve^\Omega$ yields isomorphisms for the horizontal differentials, we can take $\eta$ to be an $(m-1)$-shifted $2$-form, finishing the proof.
\end{proof}

\begin{example}
Let $H$ be a Lie group with Lie algebra $\h$ endowed with an invariant symmetric non-degenerate pairing $\kappa \in S^2\h^*$. Then, $\un{\omega}:=\kappa \in W^{2,2}(\h)=S^2\h^*$ can be seen as an infinitesimal $2$-shifted symplectic form on $A=\h$. Globally, we consider the nerve $G_\bullet=Ner_\bullet(H)$ of $H$. We observe that, when $H$ is $1$-connected, then $G$ is $2$-connected in the sense of Definition \ref{def:connected} (recall Ex. \ref{ex:1gdconnected}). Hence, from Corollary \ref{cor:shiftedsympl}, we get that $\un{\omega}=\kappa$ can be integrated to a $2$-shifted symplectic form $\omega$ on $G=Ner(H)$, $\ve^\Omega(\omega)=\kappa$. Indeed, a specific such integration $\omega$ was found in \cite[eq. (3.1) and Prop. 3.5]{cueca-zhu} and our argument can be seen as a conceptual explanation for its existence.
\end{example}


\section{Algebraic aspects of differentiation}

We now provide a simple conceptual explanation of the algebraic mechanism underlying our geometric differentiation. This approach continues and complements that of Pridham \cite{Pri10,Pri20} and connects our work with the seminal influential work of Severa \cite{Severa}.
From the abstract algebraic viewpoint, cosimplicial algebras are the global objects, differential graded algebras are the infinitesimal ones, and the differentiation is the left adjoint of the monoidal Dold-Kan denormalization. 


\subsection{Infinitesimal cosimplicial algebras}

Here, we provide explicit formulas for the denormalization in the dual Dold-Kan, introduce infinitesimal cosimplicial algebras, and derive some fundamental properties. The examples to keep in mind are the cochain cosimplicial algebra $X=C(G)$ of a simplicial manifold, and the Chevalley-Eilenberg algebra $Y=CE(A)$ of a higher Lie algebroid.

\medskip


Fix $k$ a field. A {\bf cosimplicial vector space} is a covariant functor $X:\Delta\to \cat{Vect}$, or equivalently, a system $(X^n,d^i,s^j)$ where $X^n$, $n\geq0$, are vector spaces, and
$d^i:X^{n-1}\to X^n$, $s^j:X^{n+1}\to X^n$, $0\leq i,j\leq n$, and cofaces and codegeneracies satisfying the (co)simplicial identities. Our motivating example is $C(G)$, where $G$ is a simplicial manifold, and $k=\R$. 


Given $X$ a cosimplicial vector space, its {\bf normalization} $X_N=(\bigoplus_{n\geq 0} X_N^n,d)$ is the cochain complex given by 
$X_N^n=\bigcap_{j=1}^n\ker (s^{j-1}:X^{n}\to X^{n-1})$ and with differential $d=\sum_{i=0}^{n+1} (-1)^i d^i:X_N^n\to X_N^{n+1}$.
Writing $X^n_D=\sum_{i=1}^{n}d^i(X^{n-1})\subset X^n$, there is a natural cochain complex splitting $X=X_N\oplus X_D$. 
A cochain complex is {\bf connective} if it vanishes in negative degrees.
If $f:X\to X'$ is a cosimplicial morphism, by an abuse of notation, we often write $f$ instead of $N(f):X_N\to X'_N$.


\begin{proposition}[Dual Dold-Kan correspondence]\label{prop:dual-DK}
The normalization $N$ yields an equivalence of categories 
between cosimplicial vector spaces and connective cochain complexes:
$$N:c\cat{Vect}\xto\cong Ch_{\geq0}(\cat{Vect})$$
\end{proposition}


This result is classical, although explicit formulas for the denormalization functor $K$ are seldom written out. We provide them below, following \cite{cc}, as they will help us clarify the monoidal version. While $N$ collapses redundant structure, $K$ freely recreates it out of the essential information. 

\begin{remark}[Dual DK formulas]\label{rmk:dual-denormalization}
Given $Y=(\bigoplus_{n\geq0} Y^n,d)$ a cochain complex, its {\bf denormalization} $(Y_K,d^i, s^j)$ is the cosimplicial vector space  where $Y_K^n=\bigoplus_{k=0}^n Y^k\otimes \Lambda^k(\Z^n)$.
Writing $\{e^n_1,\dots,e^n_n\}$ for the basis in $\Z^n$, $e^n_\alpha=e^n_{\alpha_1}\wedge\dots \wedge e^n_{\alpha_k}$, $\alpha_i<\alpha_{i+1}$, and $f\in Y^k$, the cofaces and codegeneracies are 
$$
s^{j-1}(f\otimes e^n_\alpha)=\begin{cases}
f\otimes e^{n-1}_{\sigma_j(\alpha)} & j \notin \alpha \\
0 & j\in \alpha
\end{cases} \quad 
d^i(f\otimes e^n_\alpha)=\begin{cases}
f\otimes e^{n+1}_{\delta_i(\alpha)} & i>0,\ i\notin \alpha\\
f\otimes (e^{n+1}_{\delta_i(\alpha)}+ e^{n+1}_{\delta_{i+1}(\alpha)}) & i>0,\ i\in \alpha\\
f\otimes e^{n+1}_{\delta_1(\alpha)} + df\otimes e^{n+1}_1\wedge e^{n+1}_{\delta_1(\alpha)} & i=0
\end{cases}$$
Writing $X=Y_K$, it follows that $f\otimes e^n_\alpha\in I_{n,j}=\ker s^{j-1}$ if and only if $j\in \alpha$, that $X_N\cong Y^n\otimes\Lambda^n(\Z^n)$, and  $X_D=\bigoplus_{k=0}^{n-1}Y^k\otimes\Lambda^k(\Z^n)$. The natural isomorphism $\phi:X\cong (X_N)_K$, is given by $\phi(f)=\sum_\alpha f_\alpha\otimes e_\alpha$,  where $f_\alpha=\sum_{\beta\subset \alpha} (-1)^\beta s^\beta(f)$.
\end{remark}


The tensor product of cosimplicial vector spaces is defined degreewise, while the tensor product of cochain complexes is the usual one. 
The dual Dold-Kan \ref{prop:dual-DK} is not a monoidal equivalence, the functors $N$ and $K$ do not preserve tensor products on the nose, even though they do so at the homotopy level, see e.g. \cite{cc}. Writing $X$ for a cosimplicial space and $Y$ for a connective cochain complex, the natural maps controlling the failure of $N$ and $K$ to be monoidal are, respectively, the {\bf Alexander-Whitney} map and the {\bf shuffle} map: 
$$aw_X:X_N\otimes X_N \to (X\otimes X)_N
\qquad 
aw_X(f\otimes g)= (d^{p+1})^q(f) \otimes (d^0)^p(g) \qquad f\in X^p_N,\ g\in X^q_N.$$
$$sh_Y:Y_K\otimes Y_K \to (Y\otimes Y)_K \qquad sh_Y((f\otimes e^n_\alpha)\otimes (g\otimes e^n_\beta))= (f\otimes g)\otimes (e^n_\alpha\wedge e^n_\beta) 
\qquad f\in Y^p,\ g\in Y^q.$$


\begin{proposition}\label{prop:overlapping}
Let $X$ be a cosimplicial vector space, $Y=X_N$ its normalization, and $\JJ_X$ the kernel of the (normalized) shuffle map $sh_{X_N}:(X\otimes X)_N\to X_N\otimes X_N$. Then:
\begin{enumerate}[a)]
    \item $(X\otimes X)_N=(Y_K\otimes Y_K)_N$ is spanned by $(f\otimes e_\alpha)\otimes(g\otimes e_\beta)$ with $(\alpha,\beta)$ covering $\{1,\dots,n\}$. 
    \item $\JJ_X\subset (X\otimes X)_N=(Y_K\otimes Y_K)_N$ is spanned by: 
    \begin{itemize}
        \item the {\bf overlapping tensors} $(f\otimes e_\alpha)\otimes(g\otimes e_\beta)$, where $(\alpha,\beta)$ is a covering of $\{1,\dots,n\}$ with overlap; and 
        \item $(f\otimes e_\alpha)\otimes(g\otimes e_\beta)-
(f\otimes e_{\tau_k(\alpha)})\otimes(g\otimes e_{\tau_k(\beta)})$,
where $\alpha,\beta$ is a partition and $\tau_k=\tau_{k,k+1}$ is an {\bf allowed transposition}, namely one such that $k,k+1$ are not both in $\alpha$ nor $\beta$.
    \end{itemize}
    \item As a differential subspace of $(X\otimes X)_N$, $\JJ_X$ is spanned just by the overlapping tensors.
    \item $sh_{X_N}$ is left inverse to $aw_{X}$, so the sequence splits, and $(X\otimes X)_N\cong (X_N\otimes X_N)\oplus \JJ_X$:
$$\xymatrix{
0 \ar[r] & \JJ_X \ar[r] & (X\otimes X)_N \ar[r]^{sh_{X_N}} &  X_N\otimes X_N \ar[r] \ar@/^1pc/[l]^{aw_X} & 0}$$
 \end{enumerate}
\end{proposition}

\begin{proof}
Claim a) is clear from our formulas for the codegeneracies given in Remark \ref{rmk:dual-denormalization}.
Regarding b), if $\alpha,\beta$ have overlap then $sh_Y((f\otimes e_\alpha)\otimes(g\otimes e_\beta))=(f\otimes g)\otimes(e_\alpha^n\wedge e^n_\beta)$ clearly vanishes, and if $\alpha,\beta$ is a partition, we have $sh_Y(f\otimes e_\alpha)\otimes(g\otimes e_\beta)=(-1)^{(\alpha,\beta)} (f\otimes g)\otimes e^{n}_{1,\dots,n}$, where $(-1)^{(\alpha,\beta)}$ is the sign corresponding to the shuffle $(\alpha,\beta)$. Then, in the case $\alpha,\beta$ define a partition, we see that the corresponding elements of $\JJ_X$ are spanned by allowed transpositions (recall Prop. \ref{prop:canonical-partition}), finishing the proof of b). Next, we prove c). If $(\alpha,\beta)$ is a covering with exactly one overlap at $j$, using Remark \ref{rmk:dual-denormalization}, we can expand
$$\delta ((f\otimes e_\alpha)\otimes(g\otimes e_\beta))=\sum_{i=0}^{n+1}(-1)^i d^i(f\otimes e_\alpha)\otimes d^i(g\otimes e_\beta)
$$
to get $(f\otimes e_\alpha)\otimes(g\otimes e_\beta)-
(f\otimes e_{\tau_k(\alpha)})\otimes(g\otimes e_{\tau_k(\beta)})$ arising from $d^j$, plus an overlapping tensor arising from $d^0$, plus terms that are not coverings, which must cancel out because the final result is normalized. This shows that allowed transposition generators are obtained through $\delta$ from overlapping tensors.
Finally, let us prove d). If $f\in Y^p$ and $g\in Y^q$, our formulas for the cofaces in Remark \ref{rmk:dual-denormalization} imply
$$(d^{p+1})^q(f\otimes e^p_{1,\dots,p})=(f\otimes e^{p+q}_{1,\dots,p})
\qquad
(d^0)^p(g\otimes e^q_{1,\dots,q})=
g\otimes e^{p+q}_{p+1,\dots,p+q} +
\sum_{i=1}^p d(g)\otimes e^{p+q}_{i,p+1,\dots,p+q}$$
It then follows that 
$aw_{Y_K}(f\otimes g)= 
f\otimes e^{p+q}_{1,\dots,p}
\otimes g\otimes e^{p+q}_{p+1,\dots,p+q} + \text{ terms in $\JJ_X$}$.
\end{proof}




Given $(X,\mu)$ a cosimplicial algebra, its normalization becomes a differential graded algebra when endowed with the {\bf cup product} $\cup=N(\mu)\circ aw_X$. 
Analogously, given $(Y,\cup)$ a differential graded algebra, $Y_K$ is a cosimplicial algebra with the {\bf wedge product} $\mu=K(\cup)\circ sh_Y$:
$$\cup=N(\mu)\circ aw_X:X_N\otimes X_N\to X_N \qquad 
f\cup g = (d^{p+1})^q(f)\cdot_\mu (d^0)^p(g) \quad f\in X_N^p,\ g\in X_N^q$$
$$\mu=K(\cup)\circ sh_Y:Y_K\otimes Y_K\to  Y_K \qquad (f\otimes e^n_\alpha)\cdot (g\otimes e^n_\beta)= (f\cup g)\otimes e^n_\alpha\wedge e^n_\beta \quad f\in Y^p,\ g\in Y^q$$
This way, we can upgrade the normalization and denormalization to functors $$N^\otimes: c\cat{Alg} \to dg\cat{Alg},\qquad K^\otimes:dg\cat{Alg}\to c\cat{Alg}$$
These functors no longer define an equivalence, nor do they form an adjoint pair. By an abuse of notation, we  write $N=N^\otimes$ and $K=K^\otimes$.
If $X$ is a cosimplicial algebra and $Y$ a differential graded algebra, the following are elementary facts regarding $X_N$ and $Y_K$:
\begin{enumerate}[a)]
\item If $Y$ is graded commutative then $Y_K$ is degree-wise commutative. On the other hand, even when $X$ is commutative, the normalization $X_N$ endowed with the cup product may fail to be (think of $C(G)$). 
\item There is a natural algebra isomorphism $(Y_K)_N=Y$, this easily follows from \ref{prop:overlapping}-c). Since $Y_K^n\to Y_K^0$ is an infinitesimal extension of degree $n$, not every cosimplicial algebra is of the form $Y_K$ (think again in $C(G)$). 
\item The functor $K$ is fully faithful, so it embeds the category of differential graded algebras within the category of cosimplicial algebras.
We say that a cosimplicial algebra $X$ is {\bf infinitesimal} if $X\cong Y_K$ for some $Y$, in which case $X_N\cong Y$ and $X\cong (X_N)_K$.
\end{enumerate}


The following proposition provides a characterization of infinitesimal cosimplicial algebras, where the products of the overlapping tensors appear as an obstruction.

\begin{proposition}\label{prop:infinitesimal-algebra}
A cosimplicial algebra $(X,\mu)$ is infinitesimal, namely $X\cong (X_N)_K$, if and only if the overlapping tensors lie in the kernel of $\mu$, namely $\mu(\JJ_X)=0$.
\end{proposition}

\begin{proof}
The canonical isomorphism of cosimplicial vector spaces $X\cong (X_N)_K$ is an algebra isomorphism if and only if the following diagram commutes
$$\xymatrix{ 
(X\otimes X)_N \ar[d]_{\mu} \ar[r]^{sh_{X_N}} & X_N\otimes X_N \ar[r]^{aw_X} & (X\otimes X)_N \ar[d]^{\mu} \\
X_N \ar@{=}[rr] & & X_N }$$
This is because the composition of the top and the right arrows is the normalization of the product in $(X_N)_K$, while the left arrow is the normalization of the original product in $X$. 
Under the natural decomposition $(X\otimes X)_N= aw_{X}(X_N\otimes X_N)\oplus \JJ_X$ discussed in Proposition \ref{prop:overlapping}, the top composition is the identity on the first component and 0 on the second component. 
Thus, the diagram commutes if and only if $\JJ_X$ is in the kernel of $\mu$, and the result follows.
\end{proof}


We close this subsection by showing a historically important example of infinitesimal cosimplicial algebras that arise in the work of Beilinson. A cosimplicial algebra $X$ is {\bf reduced} if $X^0=k$ is the base field, and is {\bf small} \cite{burgos} if it is associative, commutative, spanned by $X^0$ and $X^1$ under the cup-product, and satisfies $I_{1,1}^2=0$, where $I_{1,1}=\ker (s^0:X^1\to X^0)$. Similarly, a differential graded algebra $Y$ is {\bf reduced} if $Y^0=k$ and is {\bf small} if it is commutative and spanned by $Y^0\oplus Y^1$. It is easy to see that $K$ and $N$ preserve reduced and small algebras.

\begin{proposition}\label{prop:small-alg}
If $X$ is a reduced small cosimplicial algebra, then it is infinitesimal.
\end{proposition}

\begin{proof}
Given $f\otimes e^n_\alpha\in I_{n,j}=\ker(s^{j-1}:X^n\to X^{n-1})$, by smallness, we have
$f\otimes e^n_\alpha = f_1 \cup \dots \cup f_n$, with $f_i \in X^1$. Writing $s^0(f_j)=\lambda_j \in X^0=k$, we have
 $s^j(f\otimes e^n_\alpha)
= f_1 \cup \dots \cup s^0(f_j) \cup \dots \cup f_n 
= \lambda_j  (f_1 \cup \dots \cup f_{j-1} \cup f_{j+1} \cup \dots \cup f_n)$, so either $f\otimes e^n_\alpha=0$  or $f_j \in I_{1,1}$.
Then, given an overlapping tensor $(f\otimes e^n_\alpha)\otimes(g\otimes e^n_\beta)\in \JJ_X$ with $\alpha,\beta$ a covering with overlap at $j$, we have
$$
\mu(f\otimes e^n_\alpha)\otimes(g\otimes e^n_\beta)= (f_1 g_1) \cup \dots \cup (f_j g_j) \cup \dots \cup (f_n g_n)=0$$
for $f_j,g_j \in I_{1,1}$. It follows that $\mu(\JJ_X)=0$, so $X$ is infinitesimal by Proposition \ref{prop:infinitesimal-algebra}.
\end{proof}


\subsection{Abstract differentiation}

We introduce here the abstract version of the differentiating ideal for a cosimplicial algebra and present Theorem \ref{thm:4}, a monoidal refinement of the Dold–Kan correspondence, making differential graded algebras a reflective subcategory of cosimplicial algebras. We close by framing our result within the work by Pridham \cite{Pri10,Pri20}.

\medskip


The functor $K$ adds canonical nilpotents, representing canonical infinitesimals in geometry. It turns out that this has a left adjoint $N'$, regarding differential graded algebras as the infinitesimal counterpart of cosimplicial algebras. We describe it now. 
Given $X$ a cosimplicial algebra, write $I_{n,j}=\ker(s^{j-1}:X^n\to X^{n-1})$, so $(X_N)^n=\bigcap_{j=1}^n I_{n,j}$, and each 
$I_{n,j}$ is an ideal in $X^n$. Recall that, using $X\simeq (X_N)_K$ in $c\cat{Vect}$, $f\otimes e^n_\alpha\in I_{n,j}$ if and only if $j\in\alpha$. We write $I_{n,j}^2=I_{n,j}I_{n,j}\subset X^n$ for the square ideal. 

\begin{definition}
Given $X$ a cosimplicial algebra, its {\bf differentiating ideal} $\J_X\subset X_N$ is the differential ideal generated by $J_X^n= \sum_{j=1}^n I_{n,j}^2\cap X_N^n$. The {\bf abstract differentiation} of $X$ is, by definition, the differential graded algebra $N'(X)=X_N/\J_X$.
\end{definition}


\begin{example}
If $G$ is a simplicial manifold, then $CE(A_G)$ is the abstract differentiation of $C(G)$, while $W(A_G)$ is the abstract differentiation of $\Omega(G)$. 
\end{example}



The following shows, together with Prop. \ref{prop:infinitesimal-algebra}, that $J_X$ measures the failure of $X$ to be infinitesimal.

\begin{lemma}\label{lemma:JJJ}
Given $X$ a cosimplicial algebra, we have $\mu(\JJ_X)\subset \J_X$, and moreover, $\J_X$ is the smaller differential ideal spanned by the multiplication of the overlapping tensors $\mu(\JJ_X)$. In particular, we obtain that $X$ is infinitesimal iff $\mu(\JJ_X)=0$ iff $\J_X=0$.
\end{lemma}

\begin{proof}
If $(f\otimes e_\alpha^n)\otimes(g\otimes e_\beta^n)\in \JJ_X$ is an overlapping tensor, with $\alpha,\beta$ a covering and $j\in \alpha\cap\beta$, then $\mu((f\otimes e_\alpha^n)\otimes(g\otimes e_\beta^n))$ is clearly in $\J_X$, and by Proposition \ref{prop:overlapping}, $\mu(\JJ_X)\subset \J_X$.

To prove that $\JJ_X$ span $\J_X$, let us show that if $\JJ_X=0$ then $\J_X=0$, the general case then follows by replacing $X$ with the quotient $(X_N/\langle \JJ_X\rangle)_K$. So, assuming $\JJ_X=0$, we have that $X=(X_N)_K$ is an infinitesimal algebra by Proposition \ref{prop:infinitesimal-algebra}, and since $I_{n,j}$ is spanned by tensors $f\otimes e^n_\alpha$ with $j\in\alpha$, we have $I_{n,j}^2=0$, for 
$(f\otimes e^n_\alpha)(g\otimes e^n_\beta)=(f\cup g)\otimes e^n_\alpha\wedge e^n_\beta$, and
 $\J_X=0$.
\end{proof}


\begin{remark}
Unlike the normalization functor $N$, the abstract differentiation $N'$ sends 
degree-wise commutative algebras to graded commutative algebras. 
In fact, if $X\in c\cat{Alg}$ is commutative,
given $f\otimes e_{1,\dots,p}\in (X_N)_K^p$ and $g\otimes e_{1,\dots,q}\in (X_N)_K^q$, 
we have 
\begin{align*}
(f\otimes e_{1,\dots,p})\cup(g\otimes e_{1,\dots,q}) & = 
(f\otimes e_{1,\dots,p})(g\otimes e_{p+1,\dots,p+q}) + \text{overlapping tensors} \\
&\equiv_{\J_X} 
(-1)^{pq}(f\otimes e_{q+1,\dots,p+q})(g\otimes e_{1,\dots,q})\\
&\equiv_{\J_X}
(-1)^{pq}(g\otimes e_{1,\dots,q})(f\otimes e_{q+1,\dots,p+q})\\
&\equiv_{\J_X}
(-1)^{pq}(g\otimes e_{1,\dots,q})\cup (f\otimes e_{1,\dots,p})
\end{align*}
where we are using that $\mu(\JJ_X)\subset \J_X$, the description of the generators of $\JJ_X$ from \ref{prop:overlapping}, and that we can go from $(\{1,\dots,p\},\{p+1,\dots,p+q\})$ to 
$(\{q+1,\dots,p+q\},\{1,\dots,q\})$ through $pq$ allowed transpositions.
\end{remark}


We are now ready to present the refined monoidal version of Dold-Kan, which elucidates the algebraic mechanism behind the geometric differentiation $G\mapsto A_G$ of the previous sections. 

\begin{theorem}[Abstract differentiation]\label{thm:4}
The abstract differentiation $N'(X)=X_N/\J_X$
is left adjoint to the denormalization functor, making differential graded algebras a reflective subcategory of cosimplicial algebras.
$$N':c\cat{Alg}\leftrightarrows dg\cat{Alg}:K, \qquad N'\dashv K.$$
Moreover, $N'$ preserves commutativity, and for a simplicial manifold $G$ it recovers the geometric differentiation as $N'(C(G))\cong CE(A_G)$.
\end{theorem}

\begin{proof}
Given $X$ a cosimplicial algebra and $Y$ a differential graded algebra, we have to show that there is a natural isomorphism
$$Hom(X,Y_K)\cong Hom(X_N/\J_X,Y)$$
Given $\phi:X\to Y_K$ in $c\cat{Alg}$, $\phi(\J_X)\subset \J_{Y_K}=0$, and therefore, we have a differential graded algebra map $N'(\phi):X_N/\J_X\to Y$. 

Conversely, given $\psi:X_N/\J_X\to Y$ in $dg\cat{Alg}$, write $\psi':X_N\to Y$ for the composition with $X_N \to X_N/\J_X$, and consider $K(\psi'):(X_N)_K\to Y_K$ the corresponding map in $c\cat{Vect}$. We claim that this is an algebra morphism, or in other words, that the squares below commute:
$$\xymatrix@C=50pt{
aw_{X}(X_N\otimes X_N)
\ar[d]_{\mu} \ar[r]^{N(K(\psi')\otimes K(\psi'))}   &  aw_{Y_K}(Y\otimes Y) \ar[d]^{K(\cup)\circ sh_Y} \\
X_N \ar[r]^{N(K(\psi'))=\psi'} & Y }\qquad
\xymatrix@C=50pt{
\JJ_X
\ar[d]_{\mu} \ar[r]^{N(K(\psi')\otimes K(\psi'))}   & \JJ_Y \ar[d]^{K(\cup)\circ sh_Y} \\
X_N \ar[r]^{N(K(\psi'))=\psi'} & Y }$$
The one on the left commutes because $\psi$ is an algebra map, and the one on the right commutes because the right map is 0, and because $\mu(\JJ_X)\subset \J_X$.
\end{proof}

Our Theorem \ref{thm:4}  recovers as a special case the main result regarding Beilinson's small algebras \cite[Thm 7.3]{burgos} since every reduced small cosimplicial algebra is infinitesimal by Proposition \ref{prop:small-alg}.

\begin{corollary}
The categories of reduced small cosimplicial algebras and reduced small differential graded algebras are equivalent.
\end{corollary}


\begin{remark}
The Alexander-Whitney and shuffle maps admit super-versions, and Theorem \ref{thm:4} extends to an adjunction between coimplicial and differential graded superalgebras:
$$N': c\cat{Alg}_{\Z_2}\leftrightarrows dg\cat{Alg}_{\Z_2} :K, \qquad N'\dashv K.$$
Notice that, when $X$ is a commutative cosimplicial superalgebra, the cup product incorporates appropriate extra signs with respect to the ordinary one, as in the case of forms $X=\Omega(G)$. 
\end{remark}

Versions of the Dold-Kan adjunction between cosimplicial algebras and differential graded algebras towards differentiation already appeared in the work of Pridham  \cite{Pri10,Pri20}. We comment on its relation to our approach in the following.

\begin{remark}
In \cite[Def.~4.20]{Pri10}, more elaborate formulas for a denormalization are given in a restricted setting while, in the later preprint \cite{Pri20}, 
the construction is extended and formulas for a  left adjoint $D^*$ to denormalization are also given, later proposing them for the non-commutative case as well.
Although these formulas have a different presentation, they turn out to also generate our differentiating ideal $\J_X$. 
Yet, altogether, our approach provides both geometric and algebraic clarity to the construction. 
Pridham further proposes to differentiate simplicial supermanifolds $G$ by applying $D^*$ to the structure sheaf pulled back to $G_0$, see \cite[Sec. 4.3.1]{Pri20}.
Our approach may be viewed as a concrete and geometric refinement of this idea, where the above relations arise naturally from a universal algebraic obstruction, namely the higher vanishing ideal $J_X\subset X_N$, whose geometric origin becomes transparent through the normal form theory developed earlier in the paper. 
Within our framework, differentiation is described globally, the localization of cochains at $G_0$ is automatic modulo $J_X$ (Lemma \ref{lem:Jgenlocal}), and it does not involve a pullback to $G_0$.  
\end{remark}


\subsection{Relation with the work of Severa}\label{subsec:severa}

We finally relate our abstract differentiation to the supergeometric approach proposed by Severa \cite{Severa}. His construction formulates differentiation as a representability problem for a functor defined on a distinguished class of simplicial supermanifolds. 
As mentioned in the introduction, this problem has been previously studied in \cite{Li,LRWZ,dorsch}, where the existence of a representing graded manifold is shown. In this subsection, we show how our geometric differentiation, together with the underlying algebraic properties, provides a \emph{differential} graded manifold representing Severa's functor at once.

\medskip

Recall that a {\bf supermanifold} $M=(M_0,\O_M)$ of dimension $p|q$ is a manifold $M_0$ coupled with a sheaf of $\Z_2$-graded algebras $\O_M$ that is locally modeled by 
$\R^{p|q}=(\R^p,C^\infty(\R^p)[\epsilon_1,\dots,\epsilon_q])$.
We write $\cat{Man}_{\Z_2}$ for the category of supermanifolds, $s\cat{Man}_{\Z_2}$ for the category of simplicial supermanifolds, and $dg\cat{Man}_{\Z_2}$ for the category of differential graded supermanifolds $(M_0,\O_M)$, where $\O_M$ is now a sheaf of $\Z_2$-graded differential graded algebras that is locally modeled, as a $\Z\times\Z_2$-graded algebra, by 
$(\R^p,S_{C^\infty(\R^p)}[\xi_1,\dots,\xi_q])$, a graded symmetric algebra with generators $\xi_i$ of degree $n_i$.
A differential graded supermanifold whose parity is forced by the grading is the same as a higher Lie algebroid.


\begin{remark}
The contravariant functor $M\mapsto C^\infty(M)=\Gamma(M_0,\O_M)$ is fully faithful, hence an embedding of the category of supermanifolds into that of superalgebras. Analogously, a simplicial supermanifold gives rise to a cosimplicial superalgebra, and a dg-supermanifold gives rise to a dg-superalgebra, and these are fully faithful functors:
$$s\cat{Man}_{\Z_2}\xto{C^\infty} c\cat{Alg}_{\Z_2}$$
$$dg\cat{Man}_{\Z_2}\xto{C^\infty} dg\cat{Alg}_{\Z_2}$$
\end{remark}



The {\bf odd line} $\R^{0|1}$ plays a key role in the overall discussion together with its endomorphisms $End(\R^{0|1})=\Pi T(\R^{0|1})$. Its algebra of global sections is $C^\infty(End(\R^{0|1}))=\Omega(\R^{0|1})=\R[\epsilon,x]$, where $\epsilon$ is odd and $x$ even, $\epsilon$ of $\Z$-degree 0, and $x=d\epsilon$ for de Rham differential. Let $P(\R^{0|1})$ denote the nerve of the pair groupoid of the odd line. Given $M$ a supermanifold, we write $M\times P(\R^{0|1})$ for the product with the constant simplicial supermanifold $M$. 

\begin{proposition}\label{prop:severa-inf}
The cosimplicial algebra of cochains $C(M\times P(\R^{0|1}))$ is infinitesimal and its normalization is 
$C(M\times P(\R^{0|1}))_N\cong C^\infty(M\times End(\R^{0|1}))\cong C^\infty(M)[\epsilon,x]$. 
\end{proposition}

\begin{proof}
Let us prove the case $M=\ast$, for the general case will follow by the canonical isomorphism 
$C^\infty(M\times End(\R^{0|1}))\cong C^\infty(M)\otimes C^\infty(End(\R^{0|1}))$. 
%
%
We first verify that $I_{n,j}^2\cap X_N = 0$ so that the differentiating ideal $\hat J_X =0$ is trivial for $X=C^\infty(P(\R^{0|1}))$, namely, that $X$ is infinitesimal.

In dimension $n$ we have $P(\R^{0|1})_n=(\R^{0|1})^{n+1}=\R^{0|n+1}$, and therefore, $C^\infty(P(\R^{0|1})_n)=\R[\epsilon_0,\dots,\epsilon_n]$.  
The cofaces and codegeneracies of $C^\infty(P(\R^{0|1}))$ coming from the pair groupoid are given by $$d^i:\R[\epsilon_0,\dots,\epsilon_{n-1}]\to
\R[\epsilon_0,\dots,\epsilon_{n}] \qquad d^i(\epsilon_k)=\begin{cases}
\epsilon_k & k<i\\ \epsilon_{k+1} & k\geq i
\end{cases}$$
$$s^j:\R[\epsilon_0,\dots,\epsilon_{n+1}]\to
\R[\epsilon_0,\dots,\epsilon_{n}] \qquad s^j(\epsilon_k)=\begin{cases}
\epsilon_k & k<j\\ \epsilon_{k-1} & k\geq j
\end{cases}$$
As in ordinary geometry, a cochain has higher vanishing, namely $f\in I^2_{n,j}$, if it is a multiple of $(\epsilon_j-\epsilon_{j-1})^2$. But $(\epsilon_j-\epsilon_{j-1})^2=0$ in $\R[\epsilon_0,\dots,\epsilon_n]$, so that $\J_X=0$ as claimed.

Finally, let us verify that $C(P(\R^{0|1}))_N\cong \R[\epsilon,x]$. 
From the above, it follows that a cochain $f\in C^\infty(P(\R^{0|1})_n)$ is normalized if it is a multiple of $\Delta_n:=(\epsilon_1-\epsilon_0)\cdots(\epsilon_n-\epsilon_{n-1})$ in $ \R[\epsilon_0,\dots,\epsilon_n]$. Taking into account nilpotency of the coordinates, a general normalized cochain must take the form 
\[ f = a \pi_n + b \Delta_n, \ a,b\in \R \text{ where $\pi_n = \epsilon_0\cdots \epsilon_n$.}\]
Next, computing directly with the definitions and $d=\sum_i (-1)^i d^i$, we obtain 
$$\epsilon_i\Delta_n  = \pi_n, \ d(\pi_{n-1})=\Delta_n.$$
We thus define the isomorphism $C^\infty_N(P(\R^{0|1}))\cong\R[\epsilon][x]$ by $\Delta_n \mapsto x^n$ and $\pi_n\mapsto \epsilon x^n$. The fact that it preserves products follows directly from the relations $\pi_p\cup \pi_q=0$, $\Delta_p\cup \pi_q = \pi_{p+q}$ and $\Delta_p \cup \Delta_{q}=\Delta_{p+q}$.
\end{proof}

\begin{remark}
Recall from classical Lie theory that, for any manifold $M$, cochains on the nerve of its pair groupoid $P(M)$ differentiate to forms on $M$, $CE(A_{P(M)})\simeq \Omega(M)$.
Then, the proof of $C^\infty_N(P(\R^{0|1}))\cong\R[\epsilon,x]$ above can be seen as a particular super-case of differentiation:  since $\J_{P(\R^{0|1})}=0$, we have $C^\infty_N(P(\R^{0|1}))=CE(A_{P(\R^{0|1})})\simeq \Omega(\R^{0|1})= \R[\epsilon,x]$.
\end{remark}


Severa proposes in \cite{Severa} to consider the differentiation of a simplicial supermanifold $G$ as follows. Let $\Csev$ be the full subcategory of simplicial supermanifolds consisting of the $M\times P(\R^{0|1})$. 
Then, {\bf Severa's differentiation} is defined to be the presheaf 
$$F_G:(\Csev)^\circ\to \cat{Sets} \qquad
F_G(M\times P(\R^{0|1}))=Hom_{s\cat{Man}_{\Z_2}}(M\times  P(\R^{0|1}),G)$$
The problem of differentiating then translates into finding a differential graded supermanifold (equiv. the underlying higher Lie super-algebroid) suitably representing this functor. In the sequel, we show that our $A_G$ appropriately represents $F_G$.

\begin{lemma}
The differentiation functor canonically embeds the category $\Csev$ as the full subcategory of $dg\cat{Man}_{\Z_2}$ of products $M\times End(\R^{0|1})$:
$$Lie:\Csev\to dg\cat{Man}_{\Z_2}, \qquad Lie(M\times  P(\R^{0|1}))=M\times End(\R^{0|1}).$$
\end{lemma}

\begin{proof}
Taking global sections is fully faithful, so it is enough to prove that the full subcategory of $c\cat{Alg}_{\Z_2}$ formed by the infinitesimal algebras
$C^\infty(M\times P(\R^{0|1}))=C^\infty(M)[\epsilon,x]_K$, see Proposition \ref{prop:severa-inf}, is equivalent to the full subcategory of $dg\cat{Alg}_{\Z_2}$ formed by the objects $C^\infty(M\times End(\R^{0|1}))=C^\infty(M)[\epsilon,x]$. This immediately follows from Theorem \ref{thm:4}.
\end{proof}


We are in a position to show that our geometric differentiation provides a complete solution to the representability problem proposed by Severa.

\begin{proposition}\label{prop:Severa}
Given $G$ be a simplicial supermanifold, then Severa's differentiation presheaf $F_G$ is represented by the dg supermanifold $A_G$ defined by our geometric differentiation $CE(A_G)=C^\infty_N(G)/\J$,
in the sense that the following diagram commutes up to natural isomorphism:
$$\xymatrix{ \Csev \ar@{^(->}[r]^{Lie} \ar[rd]_{F_G} &  
dg\cat{Man}_{\Z_2} \ar[d]^{Hom(-,A_G)} \\ & \cat{Sets}
}$$
\end{proposition}

\begin{proof}
By combining Proposition \ref{prop:severa-inf} and Theorem \ref{thm:4} we get
\begin{align*}
F_G(M\times P(\R^{0|1})) &=Hom_{s\cat{Man}_{\Z_2}}(M\times P(\R^{0|1}),G) \\
& \cong
Hom_{c\cat{Alg}_{\Z_2}}(C(G), C(M\times P(\R^{0|1})))\\
& \cong
Hom_{c\cat{Alg}_{\Z_2}}(C(G), K(C^\infty(M)[\epsilon,x]))\\
& \cong
Hom_{dg\cat{Alg}_{\Z_2}}(N'C(G), C^\infty(M)[\epsilon,x])\\
& \cong
Hom_{dg\cat{Alg}_{\Z_2}}(CE(A_G), C^\infty(M\times End(\R^{0|1})))\\
&\cong
Hom_{dg\cat{Man}_{\Z_2}}(M\times End(\R^{0|1}),A_G).
\end{align*}
\end{proof}

\begin{remark}
Let us describe further the above bijection taking $\Phi\in Hom_{s\cat{Man}_{\Z_2}}(M\times P(\R^{0|1}),G)$ to a $\Phi': CE(A_G)\to C^\infty(M)[\epsilon,d\epsilon]$ corresponding to the map in $Hom_{dg\cat{Man}_{\Z_2}}(M\times End(\R^{0|1}),A_G)$. The information of the map $\Phi$ is completely encoded in the induced map between normalized cochains,
\[ \Phi^*: C^\infty_N(G) \to C^\infty(M)\otimes C^\infty_N(P(\R^{0|1})) \simeq C^\infty(M)[\epsilon,d\epsilon] = C^\infty(M)\otimes CE(A_{P(\R^{0|1})}) \]
where we used that normalized cochains already yield the infinitesimal ones on the right, by Prop. \ref{prop:severa-inf}. Then, such a map must factor through the van Est differentiation map $\ve:C^\infty_N(G)\to CE(A_G)$, $\Phi^*=\Phi' \circ \ve$, defining the corresponding dga map $\Phi'$. Using generators induced by a frame for $G$, we have $\Phi^*(x_{\iota_n,\ell})= \Phi'(\ell)$ for $\ve(x_{\iota_n,\ell})=\ell \in \Gamma A_n^*$.
\end{remark}

Following \cite{L-BS}, there is an equivalent alternative approach to Severa's representability problem, where one first seeks a supermanifold $\mathcal M_G$ representing the presheaf
$$\tilde F_G:\cat{Man}_{\Z_2}\to\cat{Sets}, \qquad
\tilde F_G(M)=
Hom_{s\cat{Man}_{\Z_2}}(M\times  P(\R^{0|1}),G)$$
and then endowed it with the NQ-structure coming from the natural $End(\R^{0|1})$-action. The two representability problems can be seen to be equivalent via an abstract general $End(\R^{0|1})$-equivariant Yoneda principle, see \cite{Severa}. 
Next we provide a concrete coordinate-based discussion showing that our $A_G$ also solves this representability problem.

We need to show that: a) $Hom_{dg\cat{Man}_{\Z_2}}(M\times End(\R^{0|1}),A_G)\simeq Hom_{Man_{\Z_2}}(M,\mathcal{M}_{A_G})$ naturally on $M$ and with $\mathcal{M}_{A_G}$ the supermanifold corresponding to $A_G$ (as in Remark \ref{rmk:superA}); b) the $NQ$-structure on $\mathcal{M}_{A_G}$ induced by the endomorphisms on $\R^{0|1}$ 
corresponds to our Chevalley-Eilenberg dga structure on $CE(A_G)=C^\infty(\M_{A_G})$.
Both these facts can be easily checked at the level of generators. Let $\varphi:S(\Gamma A_G^*) \overset{\sim}{\to} CE(A_G)$ be a splitting (e.g. $\varphi=\phi_\star$ induced by a frame on $G$) and denote $d_\varphi$ the induced differential on the domain. We clearly have 
\begin{eqnarray*} 
Hom_{dg\cat{Man}_{\Z_2}}(M\times End(\R^{0|1}),A_G)&=&Hom_{dg\cat{Alg}}(CE(A_G),C^\infty(M)[\epsilon,x])\\ &\simeq &  Hom_{dg\cat{Alg}}(S(\Gamma A^*),C^\infty(M)[\epsilon,x]) \ni \Phi'
\end{eqnarray*}
and that an element $\Phi'$ as above is characterized by its values on generators $\ell \in \Gamma A_n^*$,
$$ \Phi'(\ell) = (z_{n,\ell}+ \tilde z_{n,\ell} \ \epsilon) (d\epsilon)^n, \ \  z_{n,\ell}, \tilde z_{n,\ell}\in C^\infty(M).$$
Since $\Phi'$ commutes with differentials, we can determine $\tilde z$'s in terms of the $z$'s. 
Indeed, before differentiation (in terms of the $\Phi$ of the above Remark), this can be seen conceptually from the fact that every $n$-simplex in the pair groupoid nerve  $P(\R^{0|1})$ is the $0$-face of the $(n+1)$-simplex obtained by adding $0\in \R^{0|1}$ in the zeroth vertex. 
The coefficients $z_{n,\ell}$ can be understood as suitably parametrizing the superspace of dga maps $\Phi'$. In turn, these coefficients can be identified with $M$-families of elements in $\Gamma A_n$ with parity $n\ mod \ 2$. 
In this way, $\Phi' \mapsto (z_{n,\ell})$ induces the desired bijection in a) above.
Finally, since $\epsilon$-translations by an odd parameter $\theta$ induce $z_{n,l}\mapsto \theta \tilde z_{n,l}$ on the space of maps $\Phi'$, and unravelling the mentioned relations between $z$ and $\tilde z$ coming from $d\Phi'=\Phi' d_\varphi$, it easily follows that the $Q$-structure of b) above corresponds to $d_\varphi\simeq d_{CE(A_G)}$, as wanted. 
The $N$-grading on $CE(A_G)$ is analogously shown to coincide with the one induced by rescaling in $\R^{0|1}$, $\epsilon \mapsto \lambda \epsilon, \ \lambda>0$. See  \cite{L-BS} for a particular illustration in the context of Courant algebroids.

\frenchspacing

{

}

\

\noindent
Alejandro Cabrera\\
Instituto de Matemática, Universidade Federal do Rio de Janeiro (UFRJ), Rio de Janeiro, Brazil.\\
E-mail: alejandro@matematica.ufrj.br 

\

\noindent
Matias del Hoyo\\
Instituto de Matemática, Universidade Federal do Rio de Janeiro (UFRJ), Rio de Janeiro, Brazil.\\
E-mail: mdelhoyo@im.ufrj.br

\end{document}